\documentclass[12pt,reqno]{amsart}
\usepackage{amssymb}
\usepackage{mathrsfs}
\usepackage[all]{xy}
\numberwithin{equation}{section}

\theoremstyle{plain}
\newtheorem{prop}{Proposition}[section]

\newtheorem{theo}[prop]{Theorem}

\newtheorem{lemm}[prop]{Lemma}
\theoremstyle{remark}
\newtheorem{rema}[prop]{Remark}
\theoremstyle{definition}
\newtheorem{defi}[prop]{Definition}

\newtheorem{exam}[prop]{Example}
\numberwithin{equation}{section}

\newcommand{\A}{{\mathbb A}}

\newcommand{\PP}{{\mathbb P}}

\newcommand{\Q}{{\mathbb Q}}
\newcommand{\F}{{\mathbb F}}
\newcommand{\G}{{\mathbb G}}
\newcommand{\N}{{\mathbb N}}
\newcommand{\R}{{\mathbb R}}
\newcommand{\Z}{{\mathbb Z}}

\newcommand{\cE}{{\mathcal E}}
\newcommand{\cF}{{\mathcal F}}
\newcommand{\cG}{{\mathcal G}}
\newcommand{\cO}{{\mathcal O}}
\newcommand{\cI}{{\mathcal I}}

\newcommand{\ra}{\rightarrow}

\newcommand{\eqto}{\stackrel{\lower1.5pt\hbox{$\scriptstyle\sim\,$}}\to}
\DeclareMathOperator{\Gal}{Gal}

\DeclareMathOperator{\Pic}{Pic}
\DeclareMathOperator{\Spec}{Spec}

\DeclareMathOperator{\End}{End}
\DeclareMathOperator{\Br}{Br}

\DeclareMathOperator{\Isom}{Isom}
\DeclareMathOperator{\Proj}{Proj}
\DeclareMathOperator{\Sym}{Sym}

\DeclareMathOperator{\chara}{char}

\begin{document}
\title[Brauer-Severi surface bundles]{Models of Brauer-Severi surface bundles}
\author{Andrew Kresch}
\address{
  Institut f\"ur Mathematik,
  Universit\"at Z\"urich,
  Winterthurerstrasse 190,
  CH-8057 Z\"urich, Switzerland
}
\email{andrew.kresch@math.uzh.ch}
\author{Yuri Tschinkel}
\address{
  Courant Institute,
  251 Mercer Street,
  New York, NY 10012, USA
}

\address{Simons Foundation, 160 Fifth Av., New York, NY 10010, USA}
\email{tschinkel@cims.nyu.edu}

\date{August 20, 2017} 

\maketitle

\section{Introduction}
\label{sec:introduction}

This paper is motivated by the study of rationality properties of 
Mori fiber spaces. These are algebraic varieties, naturally occurring
in the minimal model program, typically via
contractions along extremal rays; they are
fibrations  with {\em geometrically} rational generic fiber.
Widely studied are conic bundles $\pi: X\to S$, when the generic fiber is a conic. 
According to the Sarkisov program \cite{sarkisovconicbundle}, 
there exists a birational modification

\centerline{
\xymatrix{ 
\widetilde{X} \ar_{\tilde{\pi}}[d] \ar@{-->}[r] & X\ar^{\pi}[d]\\
\widetilde{S} \ar[r] & S
}
}

\noindent
such that: (i) the general fiber of $\tilde{\pi}$ is a smooth
conic, (ii) the discriminant divisor of $\tilde{\pi}$ is a simple normal crossing divisor,
(iii) generally along the discriminant divisor, 
the fiber is a union of two lines,
and (iv) over the singular locus of the 
discriminant divisor, the fiber is a
double line in the plane.

Rationality of conic bundles over surfaces
has been classically studied by Artin and Mumford \cite{AM}, who produced
examples of nonrational unirational threefolds of this type, 
computing their Brauer groups and using its nontriviality as an obstruction to rationality. 
This was generalized to higher-dimensional quadric bundles 
in \cite{CTO}, bringing higher unramified cohomology into the subject and providing 
further examples of nonrational varieties based on these new obstructions.

The specialization method, introduced by Voisin \cite{voisin} and developed further by Colliot-Th\'el\`ene--Pirutka \cite{CTP}, emerged as a powerful tool in the study of {\em stable rationality}. It allows to relate
the failure of stable rationality of a very general member of a flat family to 
the existence of special fibers with nontrivial unramified cohomology and mild singularities; 
see also
\cite{beauvillesexticdouble}, \cite{totarojams}.
In particular, the stable rationality problem for very general 
smooth rationally connected threefolds
can be reduced to the case of conic bundles over rational surfaces 
\cite{HTfano}, \cite{krylovokada}. 
The case of very general families of conic bundles 
over rational surfaces is treated in \cite{HKTconic}.
Many families that arise naturally in applications lead to
discriminant curves of a special form. These have been studied, e.g., in \cite{BGvB} and \cite{HTfano}.

A key step in the specialization method is the construction of families of Mori fiber spaces 
with controlled degeneration.
A next case to consider is a fibration whose generic fiber is a form of $\PP^2$.

\begin{defi}
\label{defn:P2bundle}
Let $k$ be a perfect field of characteristic different from $3$
and $\pi\colon X\to S$ a morphism of smooth projective varieties
over $k$.
We call $\pi$ a \emph{standard Brauer-Severi surface bundle} if
there exists a simple normal crossing divisor $D\subset S$ whose
singular locus $D^{\mathrm{sing}}$ is smooth, such that:
\begin{itemize}
\item $\pi$ is smooth over $S\smallsetminus D$ and the generic fiber of
$\pi$ is a nontrivial form of $\PP^2$ over $k(S)$.
\item Over every geometric point of $D\smallsetminus D^{\mathrm{sing}}$ the fiber of
$\pi$ is a union of three Hirzebruch surfaces $\F_1$, meeting transversally,
such that any pair of them meets along a fiber of one and the
$(-1)$-curve of the other, while over the generic point of every irreducible component
of $D$ the fiber of $\pi$ is irreducible.
\item Over every geometric point of $D^{\mathrm{sing}}$ the fiber of
$\pi$ is an irreducible scheme whose underlying reduced subscheme is
isomorphic to the cone over a twisted cubic curve.
\end{itemize}
\end{defi}

We note that such $\pi\colon X\to S$ is necessarily flat.
There is a geometric description, which in the case of
algebraically closed base field
is due to Artin \cite{artinleft}, \cite{artinlocal}, and Maeda \cite{maeda}:

\begin{theo}
\label{thm:P2bundle}
Let $k$ be a perfect field of characteristic different from $3$
and $S$ a smooth projective algebraic variety over $k$.
Assume that over $k$ the following holds:
embedded resolution of singularities for reduced subschemes of $S$
of pure codimension $1$,
and desingularization
for reduced finite-type schemes of pure dimension equal to $\dim(S)$
that is functorial with respect to \'etale morphisms.
Let
\[
\pi\colon X\ra S
\] 
be a morphism of projective varieties over $k$, whose 
generic fiber is a nontrivial form of $\PP^2$.
Then there exists a commutative diagram
\[
\xymatrix{ 
\widetilde{X} \ar[d]_{\tilde{\pi}} \ar@{-->}[r]^{\varrho_X}& X\ar[d]^{\pi}\\
\widetilde{S} \ar^{\varrho_S}[r] & S
}
\]
where if we let $U\subset S$ denote the locus over which
$\pi$ is smooth,
\begin{itemize}
\item $\varrho_S$ is a birational morphism that restricts to an
isomorphism over $U$,
\item $\varrho_X$ is a birational map that restricts to an isomorphism
over $\varrho_S^{-1}(U)\to U$, and
\item $\tilde \pi$ is a standard Brauer-Severi surface bundle.
\end{itemize}
\end{theo}

\begin{rema}
\label{rema.surfacespositivechara}
When $\dim(S)=2$, only the embedded resolution for
curves on $S$ is required for the proof of Theorem \ref{thm:P2bundle}.
However, functorial desingularization of surfaces is known even in
positive characteristic \cite{cossartjannsensaito}.
We have decided not to mention the dimension of $S$ when stating
the hypotheses concerning resolution of singularities.
\end{rema}

In this paper we describe a technique based on root stacks,
appearing in \cite{HKTconic},
that allows us to recover the Sarkisov program and its
version by Artin and Maeda (Theorem \ref{thm:P2bundle})
and obtain an extension to more general del Pezzo surface fibrations.
The proof of Theorem \ref{thm:P2bundle} illustrates this technique, which
leads to a global version of the constructions of Artin and Maeda that is
crucial in moduli problems and in applications of specialization.
As an application to rationality problems we prove:

\begin{theo}
\label{dp2thm}
Let $S$ be a smooth del Pezzo surface of degree $2$ over an
uncountable algebraically closed field $k$ of
characteristic $\ne 2$, $3$, and let $d\ge 3$.
Then a standard Brauer-Severi surface bundle
corresponding to a very general member of the linear system
$|d(-K_S)|$ is not stably rational.
\end{theo}

\noindent {\bf Acknowledgments:} We are grateful to Brendan Hassett and Alena Pirutka 
for stimulating discussions on related topics and to
Asher Auel for helpful comments.
The first author is partially supported by the
Swiss National Science Foundation.
The second author is partially supported by NSF grant 1601912.

\section{Preliminaries}
\label{sect:prelim}

We start with some algebraic results that will be used in our approach.
Let $k$ be a field.
By a \emph{variety} over $k$ we mean a geometrically integral separated
finite-type scheme over $k$.
We will need, however, the generality of locally Noetherian schemes and
Deligne-Mumford stacks.
This additional generality allows us to construct models
of Brauer-Severi bundles in which all of the fibers are smooth.
This comes at the cost of imposing nontrivial stack structure on the base.

\subsection{Deligne-Mumford stacks}
\label{ss.DM}
Deligne and Mumford \cite{DM} defined a class of stacks which includes all stack
quotients of the form $[X/G]$, where $X$ is a variety and $G$ a finite group, and
which are now called Deligne-Mumford stacks.
The stack $[X/G]$ differs from the conventional quotient $X/G$
(which exists as a variety, e.g., when $X$ is quasi-projective) in that
it keeps track of the stabilizers of the $G$-action.
A further advantage is that $[X/G]$ has the same local properties
(smoothness, etc.) as $X$.
This will be crucial for us, as we will rely on properties of smooth varieties.
Most important for us will be \emph{orbifolds} over $k$, which are
smooth geometrically integral separated finite-type Deligne-Mumford stacks over $k$
that possess an open substack isomorphic to a scheme.

As a stack, a scheme $X$ is encoded by the category of all schemes with
morphism to $X$, and $[X/G]$, as $G$-torsors $E\to T$ with
equivariant morphism $E\to X$.
When $X$ is a point (i.e., $X=\Spec(k)$ if we work over a base field $k$,
or $\Spec(\Z)$ if we work with all schemes),
$[X/G]$ is the category $BG$ of
$G$-torsors, the \emph{classifying stack} of $G$.
We observe, the diagonal of $BG$ is a finite \'etale morphism of degree $|G|$.

An infinite discrete group $G$ is also permitted according to
modern usage of the term Deligne-Mumford stack, leading to pathologies such as
zero-dimensional Deligne-Mumford stack
$[\Spec(\overline{\Q}\otimes_{\Q}\overline{\Q})/H]$ for an index $2$
non-open subgroup $H$ of $\Gal(\overline{\Q}/\Q)$,
which exists (by \cite[Thm.\ 1]{smithwilson}), is
reduced (by \cite[4.3.5]{EGAIV}), and is irreducible with $2$ points
(objects over $\Spec(\Omega)$ for all fields $\Omega$, modulo morphisms over
$\Spec(\Omega')\to \Spec(\Omega)$ for embeddings $\Omega\to\Omega'$).

When we restrict to Noetherian Deligne-Mumford stacks, i.e.,
quasi-compact locally Noetherian Deligne-Mumford stacks with
quasi-compact diagonal, such pathologies are excluded.

More general algebraic stacks such as $BG$ for positive-dimensional
algebraic groups or nonreduced group schemes $G$ (Artin stacks) are not needed here.

\subsection{Gerbes}
\label{ss.gerbes}
Let $X$ be a Noetherian Deligne-Mumford stack and $n$ a positive integer,
invertible in the local rings of an \'etale atlas of $X$.
A gerbe over $X$ banded by roots of unity $\mu_n$, or just $\mu_n$-gerbe, is a
Deligne-Mumford stack $G$ with morphism $G\to X$ that \'etale locally over $X$
is isomorphic to a product with the classifying stack $B\mu_n$ and is equipped
with compatible identifications of the automorphism groups of local sections
with $\mu_n$.

A $\mu_n$-gerbe is classified, up to isomorphism compatible with the
identifications of automorphism groups of local sections, by a class in
$H^2(X,\mu_n)$; cf.\ \cite[\S IV.2]{milne}.

Vector bundles on a gerbe determine Brauer-Severi fibrations
(see, e.g., \cite{EHKV}).
This framework was essential in \cite{HKTconic}, in the
construction of families of conic bundles
for the application of the specialization method.
We will use the same strategy for families of
Brauer-Severi bundles of relative dimension $2$.

\subsection{Brauer groups}
\label{ss.brauer}
The Brauer group $\Br(K)$ of a field $K$ is a classical invariant,
defined as the group of similarity classes of central simple algebras over $K$.
It may also be described in terms of Galois cohomology, or in the language
of \'etale cohomology, as
\[ \Br(K)\cong H^2(\Spec(K),\G_m). \]
(All cohomology in this paper will be \'etale cohomology.)

Now let $S$ be a Noetherian scheme or Deligne-Mumford stack.
Similarity classes of sheaves of Azumaya algebras over $S$,
which naturally generalize central simple algebras over a field,
give one notion of Brauer group of $S$.
Another, more relevant for us, is the
\emph{cohomological Brauer group}
\[ \Br(S):=H^2(S,\G_m)_{\mathrm{tors}}. \]
The former is a subgroup of the latter; despite examples of a
pathological nature where they differ \cite{EHKV} they are known
to coincide in geometrically relevant cases, including all schemes that
possess an ample line bundle; see, e.g., \cite{lieblichperiodindex}
and references therein.

\begin{prop}[{\cite[Cor.\ II.1.8, II.2.2]{GB}}, {\cite[Prop.\ 2]{HKTconic}}]
\label{prop.Brbasic}
Let $S$ be a regular integral Noetherian scheme or Deligne-Mumford stack and
$S'\subset S$ a nonempty open subscheme, respectively, substack. Then:
\begin{itemize}
\item[(i)] The group $H^2(S,\G_m)$ is torsion, and the
restriction homomorphism $\Br(S)\to \Br(S')$ is injective.
\item[(ii)] If $\dim(S)\le 2$ then
every element of $\Br(S)$ is the class of a sheaf of Azumaya algebras
on $S$, and every sheaf of Azumaya algebras on $S'$ representing
an element $\alpha|_{S'}$ for some $\alpha\in \Br(S)$ extends to a
sheaf of Azumaya algebras on $S$.
\item[(iii)] If $\dim(S)=2$ and $\dim(S\smallsetminus S')=0$ then
every sheaf of Azumaya algebras on
$S'$ extends to a sheaf of Azumaya algebras on $S$, and the
restriction homomorphism $\Br(S)\to \Br(S')$ is an isomorphism.
\end{itemize}
\end{prop}

Now suppose $S$ is a smooth variety or orbifold over a field $k$, with function field 
$K=k(S)$. Then 
there is a description of the image of $\Br(S)\to\Br(K)$
in terms of residues.
In case $S$ is an orbifold,
the role of the residue field at a point is played by the
\emph{residual gerbe}
\[ \mathcal{G}_{\xi}\to \Spec(k(\xi)) \]
at a point $\xi$ of $S$ (cf.\ \cite[\S11]{LMB}, \cite[App.\ B]{rydhetaledevissage}).

\begin{prop}
\label{prop.Br}
Let $k$ be a field, $S$ a smooth variety or orbifold over $k$ with
function field $K$, and
$n$ a positive integer that is invertible in $k$.
Then there is a residue map from the $n$-torsion of
the Brauer group of $K$ to the direct sum
of $H^1$ groups of residual gerbes $\mathcal{G}_{\xi}$
at codimension $1$ points $\xi\in S^{(1)}$, which fits into an exact sequence
\[ 0\to \Br(S)[n]\to \Br(K)[n]\to
\bigoplus_{\xi\in S^{(1)}} H^1(\mathcal{G}_{\xi},\Z/n\Z). \]
\end{prop}

\begin{proof}
We already know that $\Br(S)\to \Br(K)$ is injective.
Cohomological purity \cite[\S XVI.3]{SGA4t3}, combined with the local-to-global
spectral sequence for cohomology with supports \cite[\S VI.5]{milne}, yields exact sequences
\begin{align*}
\dots\to H^{j-1}&(S\smallsetminus T,\Z/n\Z(i))\to H^{j-2c}(T,\Z/n\Z(i-c))\\ &\to
H^j(S,\Z/n\Z(i))\to H^j(S\smallsetminus T,\Z/n\Z(i))\to\dots
\end{align*}
for all $i$
when $T\subset S$ is a closed substack of pure codimension $c$ that is also smooth over $k$.
We may assume, as in \cite[Rmk.\ 4.7]{blochogus}, that $k$ is a perfect field,
and as in \cite{GB} define the residue map and prove exactness.
\end{proof}

\begin{rema}
\label{rem.purity}
When $S$ is a smooth variety the exact sequence is well known, cf.\ \cite{ct-pure},
and may be extended with a further residue map as
a complex (Bloch-Ogus complex)
\[ \bigoplus_{\xi\in S^{(1)}} H^1(k(\xi),\Z/n\Z)
\to \bigoplus_{\xi\in S^{(2)}} H^0(k(\xi),\mu_n^{-1}). \]
The complex with this additional residue map is exact
when $k$ is algebraically closed and
$S$ is a smooth projective rational surface.
\end{rema}

\subsection{Root stacks}
\label{ss.rootstacks}
We recall and describe basic properties of root stacks.
Applied to a smooth variety and smooth divisor, the root stack construction produces an orbifold.
The main function of a root stack is to
remove ramification (e.g., of a Brauer class).
In geometry, we find several advantages:
\begin{itemize}
\item We are able to relate smooth families over a root stack to
flat families over the underlying scheme.
\item The classification of smooth families may be more rigid,
e.g., for conic bundles over a surface, the general ones correspond to
Brauer classes at the generic point, while the smooth ones correspond to
unramified Brauer classes.
\item Other applications such as those requiring deformation theory also
benefit from working with a proper Deligne-Mumford stack,
instead of a quasi-projective variety.
\end{itemize}

Let $S$ be a regular integral scheme,
$D\subset S$ an effective divisor, and $n$ a positive integer,
invertible in the local rings of $S$.
Then there is the corresponding \emph{root stack} $\sqrt[n]{(S,D)}$,
an integral locally Noetherian Deligne-Mumford stack,
regular if and only if $D$ is regular \cite[\S 2]{cadman}, \cite[App.\ B]{AGV}.
Above $D$ is the \emph{gerbe of the root stack} $\mathcal{G}_D$ \cite[Def.\ 2.4.4]{cadman},
an effective divisor on $\sqrt[n]{(S,D)}$ fitting into a commutative diagram
\begin{equation}
\begin{split}
\label{eqn.diagramrootstack}
\xymatrix{
\mathcal{G}_D \ar[r] \ar[d] & \sqrt[n]{(S,D)}\ar[d] \\
D \ar[r] & S
}
\end{split}
\end{equation}
where the right-hand morphism is flat and is ramified over $D$ (so the
diagram is not cartesian unless $D$ is empty) and restricts to an
isomorphism over $S\smallsetminus D$.
The left-hand morphism makes $\mathcal{G}_D$ into a
$\mu_n$-gerbe over $D$.

\begin{rema}
\label{rem.ramification}
In the situation of Proposition \ref{prop.Br} let $\xi$ be the generic point
of an irreducible regular divisor $D$.
The ramification of the right-hand map in \eqref{eqn.diagramrootstack}
leads to a factor of $n$ in the residue map
for the generic point of the gerbe of the root stack, as compared with
that of $\xi$; cf.\ \cite[Thm.\ 10.4]{saltmanlectures}.
So, as observed by Lieblich \cite[\S 3.2]{lieblicharithmeticsurface},
the residue map for the generic point of the gerbe of the root stack
vanishes on $\Br(K)[n]$.
\end{rema}

Now let $D'\subset S$ be another effective divisor.
The fiber product $\sqrt[n]{(S,D)}\times_S\sqrt[n]{(S,D')}$ of the two
associated root stacks is the \emph{iterated root stack} $\sqrt[n]{(S,\{D,D'\})}$
\cite[Def.\ 2.2.4]{cadman}.
If $D$ and $D'$ are regular and intersect transversally, then
$\sqrt[n]{(S,\{D,D'\})}$ is regular; the same is valid with any
number of divisors.
A generalization (not used in this paper) is the case of a divisor $D$ that
\'etale locally may
be written as a union of regular divisors meeting transversally,
for which a construction by Matsuki and Olsson \cite{matsukiolsson}
based on log structures gives rise to a regular Deligne-Mumford stack,
\'etale locally isomorphic to an iterated root stack.

Locally on $S$ the divisor
$D$ is defined by the vanishing of a regular function $f$, and
then $\sqrt[n]{(S,D)}$ is the quotient stack of $\Spec(\cO_S[t]/(t^n-f))$ by
the scalar action of $\mu_n$ on $t$.
The gerbe of the root stack is, locally, defined by the vanishing of $t$, and its points
are just the points of $D$, with $\mu_n$-stabilizer,
i.e., above any point $x\in D$ with residue field $\kappa(x)$ the
gerbe of the root stack takes the form of a classifying stack $BG$
for $G=\mu_{n,\kappa(x)}$.

Often we consider flat families of projective varieties, or schemes,
and we benefit by being able to study such families over a root stack,
i.e., flat projective morphisms $P\to \sqrt[n]{(S,D)}$ of
Deligne-Mumford stacks.
The fiber over a point $s\in D$, where the gerbe of the root stack
takes the form $BG$ for $G=\mu_{n,\kappa(s)}$ as mentioned above,
has the form $[P_s/G]$ for an action of $G$ on a projective scheme $P_s$
over $\kappa(s)$.
Typically $G$ acts with stabilizers ($\mu_n$ or subgroups of $\mu_n$),
and we will speak of the locus with $\mu_n$-stabilizer
for the appropriate closed subscheme of $P_s$.

\subsection{Coarse moduli spaces}
\label{ss.cms}
Let $X$ be a Deligne-Mumford stack, and assume that $X$ has
\emph{finite stabilizer}, meaning that projection to $X$ from
$X\times_{X\times X}X$ is a finite morphism.
Then there is a \emph{coarse moduli space}, an algebraic space $Q$
with separated morphism $X\to Q$, universal for morphisms to algebraic spaces.
For instance, $X$ could be $[\Spec(A)/G]$ for some finite group $G$ acting on
affine scheme $\Spec(A)$, and then $X$ has coarse moduli space $\Spec(A^G)$,
where $A^G$ denotes the $G$-invariant subring of $A$.
The article by Keel and Mori \cite{keelmori} gives the
classic treatment of coarse moduli space in the setting where
$X$ is locally of finite type over a locally Noetherian base scheme;
this suffices for our purposes, and in this case $Q$ has the
same property and $X\to Q$ is proper and quasi-finite.
The treatment without finiteness hypothesis is due to Rydh \cite{rydhquotients}.
(These treatments allow, more generally, $X$ to be
an Artin stack with finite stabilizer.)

The construction of coarse moduli space shows that $X\to Q$ has the
\'etale local (over $Q$) form $[\Spec(A)/G]\to \Spec(A^G)$, as
described above.
In fact, $G$ may be taken to be the geometric stabilizer group at a point
of $X$, and the assertion is valid on an \'etale neighborhood of the
corresponding point of $Q$; cf.\ \cite[Lem.\ 3.4]{alperkresch}.

For the next statement we assume that $X$ is \emph{tame},
meaning that besides the finite stabilizer hypothesis we also require the
order of the geometric stabilizer group at any point
$\Spec(\Omega)\to X$ not to be divisible by the
characteristic of $\Omega$ \cite{olssonstarr}.
(This notion is also available for
Artin stacks with finite stabilizer \cite{aovtame}.)
We suppose furthermore that $X$, and hence as well $Q$, is Noetherian.
Then by \cite[Thm.\ 10.3]{alpergood},
pullback by $X\to Q$ identifies the category of vector bundles on
$Q$ with the full subcategory of vector bundles on $X$ with trivial
actions of geometric stabilizer groups at closed points.

The most relevant example
of coarse moduli space for us is a root stack
$\sqrt[n]{(S,D)}\to S$ (or an iterated root stack), and this is tame.
The next result concerns schemes, \emph{Fano over $S$} for a
Noetherian scheme $S$, by which we mean the full subcategory of schemes
over $S$, in which an object is a flat projective morphism $Z\to S$ with
Gorenstein fibers such that $\omega^\vee_{Z/S}$ is relatively ample.
The last condition is equivalent by \cite[4.7.1]{EGAIII} to the
ampleness of $\omega^\vee_{Z_s}$ on the fiber $Z_s$ for every
closed point $s\in S$.
We may also take $S$ to be an algebraic space, in which case
$Z$ is also allowed to be an algebraic space.
We may allow $S$ to be a DM stack, in which case
$Z$ is also a DM stack.
But since stacks form a $2$-category, we need to take care to obtain
an ordinary category:
a morphism from $f\colon Z\to S$ to $f'\colon Z'\to S$ is defined to be an
equivalence class of pairs $(\varphi,\alpha)$ where
$\varphi\colon Z\to Z'$ is a morphism and
$\alpha\colon f\Rightarrow f'\circ\varphi$ is a $2$-morphism,
with $(\varphi,\alpha)\sim(\tilde\varphi,\tilde\alpha)$ when there
exists a (necessarily unique) $2$-morphism
$\beta\colon \varphi\Rightarrow\tilde\varphi$, such that
$f'(\beta(z))\circ\alpha(z)=\tilde\alpha(z)$ for every object $z$ of $Z$.

\begin{prop}
\label{prop.coarsemodulispace}
Let $X$ be a tame Noetherian DM stack with coarse moduli space $X\to Q$.
Then the category of schemes, Fano over $Q$, is equivalent, by base change,
to the full subcategory of DM stacks, Fano over $X$, with trivial
actions of geometric stabilizer groups at closed points.
\end{prop}

\begin{proof}
The universal property of the coarse moduli space tells us that
base change by $X\to Q$ is a fully faithful functor.

To an object $f\colon Z\to X$ of the category of DM stacks, Fano over $X$,
there is an associated quasi-coherent sheaf of
graded $\cO_{X}$-algebras
$\bigoplus_{n\ge 0}f_*((\omega^\vee_{Z/X})^n)$.
The graded components are coherent and for sufficiently large $n$ are
locally free with formation commuting with base change to
(geometric) fibers.
If the actions of geometric stabilizer groups at closed points is trivial,
then $f_*((\omega^\vee_{Z/X})^n)$ for $n\gg 0$ satisfies the condition
stated above to descend to a locally free coherent sheaf on $Q$.
The algebra structure descends as well, and by applying $\Proj$ we obtain
an algebraic space $Y$, Fano over $Q$, which upon base change to $X$
recovers $Z\to X$ up to an isomorphism.
\end{proof}

\begin{rema}
\label{rem.cms}
As a consequence, the algebraic space $Y$, Fano over $Q$ that we obtain in
the proof of Proposition \ref{prop.coarsemodulispace}, is the
coarse moduli space of $Z$.
\end{rema}

\subsection{\'Etale local uniqueness in smooth families}
\label{ss.henselian}
In certain geometric situations such as fibrations in projective spaces,
a smooth family has a unique \'etale local isomorphism type.
The same holds $G$-equivariantly, when $G$ is a finite group
whose order is not divisible by the residue characteristic.

\begin{lemm}
\label{lem.henselianfinitegroup}
Let $A$ be a Henselian ring with residue field $\kappa$,
and $X$ a scheme, smooth over $\Spec(A)$.
Let $G$ be a finite group, whose order is invertible in $\kappa$, with
compatible actions on $X$ and on $A$ and trivial action on $\kappa$,
and $x$ a $\kappa$-point of the fiber $X_\kappa$ over the closed point of $\Spec(A)$ that is
fixed for the group action.
Then $X$ admits a $G$-invariant $A$-point that specializes to $x$.
\end{lemm}

\begin{proof}
Since $A$ and hence as well $A^G$ is Henselian, there
is a free finite-rank $A^G$-module with $G$-action,
whose base change to $\kappa$ is
isomorphic to any given finite-dimensional $\kappa$-vector space with
$G$-action.

Replacing, as needed,
$X$ by intersections of translates under $G$
of an affine neighborhood of $x$,
we may suppose that $X=\Spec(B)$ is affine.
Let $\mathfrak{m}$ denote the maximal ideal corresponding to $x$.
The $A$-linear map
\[ \mathfrak{m}\to \Omega^1_{{X_\kappa},x} \]
(sending $f$ to $df$) is surjective and $G$-equivariant.
There is therefore a lift of an appropriate $\kappa$-basis of
$\Omega^1_{{X_\kappa},x}$ to elements of $\mathfrak{m}$ whose
$A^G$-linear span is invariant under the $G$-action.
According to \cite[17.16.3(i)]{EGAIV} (really, its proof), the lifts
define a $G$-invariant subscheme $W\subset X$ that is \'etale
over $\Spec(A)$ at $x$.
Since $A$ is Henselian, there is a unique $G$-invariant $A$-point of $W$ that
specializes to $x$.
\end{proof}

\begin{lemm}
\label{lem.liftisohenselian}
Let $A$ be a Henselian ring with residue field $\kappa$, and $X$ and $Y$
schemes, smooth and projective over $\Spec(A)$.
Let $G$ be a finite group, whose order is invertible in $\kappa$, with
compatible actions on $X$, on $Y$, and on $A$, with trivial action on $\kappa$.
Let $\varphi\colon X_\kappa\to Y_\kappa$ be a $G$-equivariant isomorphism of
fibers.
If the tangent bundle $T_{X_\kappa}$ satisfies
\[ H^1(X_\kappa,T_{X_\kappa})=0, \]
then there exists a $G$-equivariant isomorphism $X\to Y$
over $\Spec(A)$ that specializes to $\varphi$.
\end{lemm}

\begin{proof}
There is a scheme $\Isom_A(X,Y)$, parametrizing isomorphisms of fibers
\cite[\S 4c]{TDTEIV}, with $G$-action,
smooth over $\Spec(A)$ at $\varphi$ by \cite[Cor.\ 5.4]{TDTEIV}.
We conclude by Lemma \ref{lem.henselianfinitegroup}.
\end{proof}

\subsection{Components in families}
The irreducible components of fibers in a flat family with
geometrically reduced fibers may be accessed by
excluding the relative singular locus and forming the sheaf of connected
components.
If $\pi\colon X\to S$ is a finitely presented smooth morphism of schemes, then
geometric components of fibers determine an \'etale equivalence relation on $X$ over $S$
whose algebraic space quotient
$\pi_0(X/S)$ is \'etale over $S$ and is characterized as
factorizing $\pi$,
\[ X\to \pi_0(X/S)\to S, \]
such that $X\to \pi_0(X/S)$ is surjective with geometrically connected fibers
\cite[6.8]{LMB}.
There are easy examples for which
$\pi_0(X/S)$ is not a scheme, e.g.,
the sum of squares of coordinates
$\A^2_\R\smallsetminus\{(0,0)\}\to \A^1_\R$.
Formation of $\pi_0(X/S)$ commutes with base change by an arbitrary
morphism $S'\to S$.

A subset of $X$, which is the pre-image of a subset of $\pi_0(X/S)$,
will be called \emph{component-saturated}.

\begin{lemm}
\label{lem.pi0}
Let $\pi\colon X\to S$ be a proper flat finitely presented pure-dimensional
morphism of schemes, such that for every $s\in S$ the smooth locus of the fiber $X_s$
is dense in $X_s$.
Then the morphism $\pi_0(X^{\mathrm{sm}}/S)\to S$ is universally closed,
where $X^{\mathrm{sm}}$ denotes the relative smooth locus of $\pi$.
If, moreover, the geometric fibers of $\pi$ all have $d$ irreducible components
for some integer $d$, then $\pi_0(X^{\mathrm{sm}}/S)\to S$ is finite \'etale
of degree $d$.
\end{lemm}

\begin{proof}
The second assertion follows from the first by the general fact that
a finitely presented universally closed \'etale morphism $T\to S$ with fibers
all of the same degree is finite.
This comes down to topological properties of the relative diagonal
(it is an open immersion
since $T\to S$ is \'etale, and we need to show that it is closed as well),
so the verification is reduced by \cite[1.10.1]{EGAIV} to the case
that $S$ is $\Spec$ of a local ring, and by \cite[2.6.2]{EGAIV}, to the
case of a strictly Henselian ring.
Each of the $d$ pre-images of the closed point then lifts to a section.
The sections cover $T$, since the image of their complement is closed
but does not contain the closed point of $S$.
So, $T$ is the disjoint union of the sections.

Now suppose that $S=\Spec(A)$, where $A$ is a local ring.
Let $z\in X$ be a maximal point of a fiber of $\pi$ and
$Z$ its closure in $X$.
We claim that $\pi(Z\cap X^{\mathrm{sm}})$ contains the closed point of $S$.
Indeed, by semicontinuity of fiber dimension \cite[13.1.5]{EGAIV}, $Z$
contains a maximal point of the fiber over the closed point,
which by hypothesis lies in $X^{\mathrm{sm}}$.

Let $n\ge 0$.
In a proper flat family of schemes of relative dimension $n$ over a
Noetherian base scheme,
the condition on fibers to have singular locus of dimension less than $n$ is open
by semicontinuity of fiber dimension.
With the local form of semicontinuity of fiber dimension \cite[13.1.3]{EGAIV}
(semicontinuity of $w\mapsto \min \dim(\pi^{-1}(\pi(w))\cap U)$
where the minimum is over open $U$ containing $w$)
and flatness (particularly, \cite[2.3.4]{EGAIV}), we see
that the condition to have pure dimension $n$ is closed.
We reduce the verification of the first assertion immediately to the case when
$S$ is affine and reduced, and with the preceding observations and
\cite[8.10.5, 11.2.7]{EGAIV} to the case when $S$ is Noetherian.
Then, as noted in the proof of \cite[Prop.\ 2.2]{kreschflattening},
the claim implies the assertion.
(In fact, the case $A$ is a DVR suffices.)
\end{proof}

\begin{rema}
\label{rem.pi0}
With the notation of Lemma \ref{lem.pi0}, any component-saturated
closed subset $W\subset X^{\mathrm{sm}}$ has the property that
formation of the closure $\overline{W}$ in $X$ commutes with base change by an
arbitrary morphism $S'\to S$.
Since an equality of sets may be verified fiberwise,
the assertion reduces to the case $S'=\Spec(\kappa(s))$,
where $\kappa(s)$ is the residue field of a point $s\in S$.
We may suppose $\pi$ has constant fiber dimension $n$.
For $x\in \overline{W}\cap \pi^{-1}(s)$,
semicontinuity of fiber dimension yields
$\dim(\pi^{-1}(s)\cap U)=n$ for all open $U\subset \overline{W}$
containing $x$.
Hence there is a maximal point of $\pi^{-1}(s)$ that lies in $W$
(equivalently, lies in $\overline{W}$) and specializes to $x$.
\end{rema}

\section{Conic bundles}
\label{sec:conicbundles}
In this section we present an ingredient from the
Sarkisov program, which gives a description of conic bundles with
regular discriminant divisor, using the
language of root stacks.

By a conic bundle over a locally Noetherian scheme $S$, in which $2$ is invertible
in the local rings,
we mean a flat projective morphism
$\pi\colon X\to S$ such that the fiber over every point of $S$ is isomorphic
to a possibly singular, possibly non-reduced conic in $\PP^2$, i.e.,
the zero-loci of a nontrivial ternary quadratic form.
We say that a conic bundle $\pi\colon X\to S$ has
\emph{mild degeneration} if all fibers are reduced.

\begin{prop}
\label{prop.coniccontraction}
Let $S$ be a regular scheme, such that $2$ is invertible in the local rings of $S$,
and let $D\subset S$ be a regular divisor.
Then the operations described below identify, up to unique isomorphism:
\begin{itemize}
\item
smooth $\PP^1$-fibrations over $\sqrt{(S,D)}$ having nontrivial $\mu_2$-actions on fibers over the
gerbe of the root stack, with
\item
conic bundles $\pi\colon X\to S$ such that $X$ is regular and the
fiber $\pi^{-1}(s)$ over a point $s\in S$ is singular if and only if $s\in D$;
all such conic bundles have mild degeneration.
\end{itemize}
Given a smooth $\PP^1$-fibration $P\to \sqrt{(S,D)}$ with
nontrivial $\mu_2$-actions on fibers over the gerbe of the root stack,
a conic bundle $\pi\colon X\to S$ with mild degeneration is obtained by
\begin{itemize}
\item 
blowing up the locus with $\mu_2$-stabilizer,
\item 
collapsing the ``middle components'' of fibers over $D$, and 
\item 
descending to $S$.
\end{itemize}
Given a conic bundle $\pi\colon X\to S$
we obtain the associated smooth $\PP^1$-fibration over $\sqrt{(S,D)}$ (up to
unique isomorphism) by
\begin{itemize}
\item
pulling back to $\sqrt{(S,D)}$,
\item 
blowing up the relative singular locus over the gerbe of the root stack, and
\item 
collapsing the ``end components'' of fibers over $D$.
\end{itemize}
\end{prop}

Proposition \ref{prop.coniccontraction} deals with flat projective families of
reduced curves with at most nodes as singularities.
This is the setting of (pre)stable curves, for which there are known
constructions to perform
the contractions indicated in the statement of Proposition \ref{prop.coniccontraction}.
Collapsing the ``middle components'' of a flat family of at-most-$3$-component
genus $0$ prestable curves $\pi\colon C\to T$ is achieved by the relative
anticanonical model $\Proj\big(\bigoplus_{m\ge 0}\pi_*((\omega_{C/T}^\vee)^m)\big)$.
The setting where ``end components'' are to be collapsed is one of
families with only $1$- and $3$-component genus $0$ prestable curves,
and \'etale locally there exists a section $\sigma$ which avoids the end components
of $3$-component fibers;
then $\Proj\big(\bigoplus_{m\ge 0}\pi_*(\cO_C(\sigma)^m)\big)$ gives the
desired contraction, and although sections exist only \'etale locally,
the construction may be performed globally over the base, as described in
\cite{kreschflattening}.

\begin{proof}
Let $P\to \sqrt{(S,D)}$ be a
smooth $\PP^1$-fibration,
such that $\mu_2$ acts nontrivially on $P_s$ for every $s\in D$ in the
notation from the discussion of stabilizers from \S\ref{ss.rootstacks}.
Let $\Spec(A)$ be an affine neighborhood in $S$ of a point $s\in D$, on which
$D$ is defined by the vanishing of a function $f\in A$.
Setting
\[ A':=A[t]/(t^2-f), \]
then, we have an open substack of $\sqrt{(S,D)}$ isomorphic to
\[
[\Spec(A')/\mu_2],
\]
where $\mu_2$ acts by scalar multiplication on $t$.
We claim that there is an \'etale neighborhood $\Spec(B)\to \Spec(A)$ of $s$ such
that for the rank $2$ vector bundle and projectivization
\[ V:=[\Spec(A'[u,v])/\mu_2]\qquad\text{and}\qquad \widehat{P}:=\PP(V), \]
where $\mu_2$ acts by scalar multiplication on $u$ and trivially on $v$,
we have
\[ \Spec(B)\times_{S}P\cong 
\Spec(B)\times_{S}\widehat{P}. \]
Equivalently,
\[ \Spec(A^{sh})\times_SP\cong 
\Spec(A^{sh})\times_{S}\widehat{P}, \]
where $A^{sh}$ is a strict Henselization of the local ring of $\Spec(A)$ 
at $s$.
By \cite[18.8.10]{EGAIV}, we have
\[ A^{sh}\otimes_AA'\cong A'^{sh}, \]
where $A'^{sh}$ is an analogous strict Henselization of $A'$.
So, it suffices to verify that
\[
\Spec(A'^{sh})\times_{\sqrt{(S,D)}}P\qquad \text{and}\qquad \Spec(A'^{sh})\times_{\sqrt{(S,D)}}\widehat{P}
\]
are equivariantly isomorphic.
By Lemma \ref{lem.liftisohenselian} it is enough
to check that the respective fibers over the closed point of
$\Spec(A'^{sh})$ are equivariantly isomorphic, and this is so
by the hypothesis of nontrivial $\mu_2$-action on fibers
of $P$ over the gerbe of the root stack.
In particular, the stabilizer locus of $P$ is a
degree $2$ \'etale representable cover of the gerbe of the root stack.

The blow-up $P_1$ of the locus of $P$ with $\mu_2$-stabilizer therefore
has geometric fibers over points of $D$ with three irreducible components and
nontrivial $\mu_2$-action only on the ``middle components''.
After collapsing ``middle components'' to $P_2$, the
geometric fibers have two irreducible components over points of $D$ and
trivial $\mu_2$-action.
By Proposition \ref{prop.coarsemodulispace}, $P_2\to \sqrt{(S,D)}$ descends to a
conic bundle $P_3\to S$.
To see that $P_1$ is the blow-up of $P_2$ along the relative singular
locus over the gerbe of the root stack, we use the fact that the
constructions described in the statement are compatible with base change
by an \'etale morphism to $S$ and carry out the constructions with
$\widehat{P}$:
\begin{align*}
\widehat{P}_1&=[\Proj(A'[p,q,r,w]/(pr-tq^2,pw-t^2qr,qw-tr^2))/\mu_2],\\
\widehat{P}_2&=[\Proj(A'[x,y,z]/(xz-t^2y^2))/\mu_2],\\
\widehat{P}_3&=\Proj(A[x,y,z]/(xz-fy^2)).
\end{align*}
Since $\widehat{P}_3$ is regular, so is $P_3$.
The steps, performed in reverse starting from $P_3$, yield $P$
up to unique isomorphism.

A conic bundle $X\to S$ with $X$ regular and
singular fibers precisely over points of $D$ has, locally,
defining equation $fx^2+ay^2+bz^2=0$ for
units $a$ and $b$, as we see using
diagonalizability of a quadratic form
over $\cO_{S,s}$, first with $s$ the generic point of a component of $D$ and
then with arbitrary $s\in D$.
(Such an observation for quadric bundles is given in \cite[Prop.\ 1.2.5]{auelbernardarabolognesi}.)
So, \'etale locally, there is an isomorphism with $\widehat{P}_3$,
and the construction may be carried out to yield a
smooth $\PP^1$-bundle over $\sqrt{(S,D)}$.
From this, the construction of blowing up, collapsing, and descending to $S$
recovers $X\to S$ up to unique isomorphism.
\end{proof}

\begin{rema}
\label{rem.goodmodelsconicbundles}
The Sarkisov program establishes the existence of good models of
conic bundles over arbitrary varieties \cite{sarkisovconicbundle},
via classical birational geometry.
A more general version of this result may
be obtained through the use of root stacks.
Details are presented in \cite{oesinghaus}.
\end{rema}

\section{Brauer-Severi surface bundles}
\label{sec.brauerseverisurfacebundles}
In the following sections we prove Theorem \ref{thm:P2bundle}
using root stacks.

\begin{defi}
\label{defn.twodimbrauerseveri}
Let $S$ be a locally Noetherian scheme, in which $3$ is invertible in the
local rings.
A \emph{Brauer-Severi surface bundle} over $S$ is a flat projective
morphism $\pi\colon X\to S$ such that the fiber over every geometric point of $S$
is isomorphic to one of the following:
\begin{itemize}
\item a del Pezzo surface of degree $9$ (i.e., since a
geometric fiber, $\PP^2$),
\item the union of three Hirzebruch surfaces $\F_1$,
meeting transversally, such that any pair of them
meets along a fiber of one and the $(-1)$-curve of the other,
\item an irreducible scheme whose underlying reduced subscheme is isomorphic
to the cone over a twisted cubic curve.
\end{itemize}
If only fibers of the first two types appear, then we say that the
Brauer-Severi surface bundle has \emph{mild degeneration}.
\end{defi}

\begin{rema}
\label{rem.10sections}
For the fibers $X_s$ appearing in a Brauer-Severi surface bundle
with mild degeneration, the dual $\omega^\vee_{X_s}$ of the
dualizing sheaf is very ample and we have $h^0(X_s,\omega^\vee_{X_s})=10$.
This is standard for del Pezzo surfaces of degree $9$ and may be verified
with a short computation for unions of three Hirzebruch surfaces.
\end{rema}

\begin{defi}
\label{def.balancedaction}
Let $k$ be a field of characteristic different from $3$.
An action of $\mu_3$ on $\PP^2_k$ is said to be \emph{balanced} if,
after passing to an algebraic closure $\bar k$ of $k$ and making
a change of projective coordinates, the action is
$\mu_3\times \PP^2_{\bar k}\to \PP^2_{\bar k}$,
\[ (\alpha,x:y:z)\mapsto (\alpha x:\alpha^2y:z). \]
\end{defi}

We formulate an analogue of Proposition \ref{prop.coniccontraction}.

\begin{prop}
\label{prop.DP9contraction}
Let $S$ be a regular scheme, such that $3$ is invertible in the local rings
of $S$, and let $D\subset S$ be a regular divisor.
Then the operations described below
identify, up to unique isomorphism:
\begin{itemize}
\item smooth $\PP^2$-fibrations over $\sqrt[3]{(S,D)}$
having balanced $\mu_3$-actions on fibers over the gerbe of the root stack, with
\item Brauer-Severi surface bundles $\pi\colon X\to S$
such that $X$ is regular and the fiber $\pi^{-1}(s)$ over
a point $s\in S$ is singular if and only if $s\in D$; all such
bundles have mild degeneration.
\end{itemize}
Given a smooth $\PP^2$-fibration $P\to \sqrt[3]{(S,D)}$ with
balanced $\mu_3$-actions on fibers over the gerbe of the root stack,
$P$ has distinguished loci 
\begin{itemize}
\item $P_\mu$: locus with nontrivial $\mu_3$-stabilizer,
\item $P_\lambda$:  union of lines on the geometric $\PP^2$-fibers
joining points of $P_\mu$.
\end{itemize}
On the blow-up  $B\ell_{P_\mu}P$ of $P$ we have a distinguished locus
\begin{itemize}
\item $P_c$: the union of codimension $2$ components of the
stabilizer locus.
\end{itemize}
On the second blowup $B\ell_{P_c}(B\ell_{P_\mu}P)$ we have the locus
\begin{itemize}
\item $\widetilde{P}_\lambda$: the proper transform of $P_\lambda$.
\end{itemize}
Then a Brauer-Severi surface bundle $\pi: X\ra S$ with mild degeneration
results by 
\begin{itemize}
\item 
collapsing first the degree $8$ components of fibers over $D$ of
\[ B\ell_{\widetilde{P}_\lambda}(B\ell_{P_c}(B\ell_{P_\mu}P)), \]
\item 
collapsing the degree $3$ components, 
\item 
descending to $S$.
\end{itemize}

Given a Brauer-Severi surface bundle $\pi:X\ra S$ we
obtain the associated smooth $\PP^2$-fibration over the root stack
$\sqrt[3]{(S,D)}$ (up to unique isomorphism)
by 
\begin{itemize}
\item pulling back to $\sqrt[3]{(S,D)},$
\item 
blowing up the relative singular locus of the
relative singular locus over the gerbe of the root stack,
\item 
blowing up the relative singular locus over the gerbe of the root stack,
and 
\item contracting the high-degree ($\ge 7$) components above the
gerbe of the root stack.
\end{itemize}
\end{prop}

The constructions in Proposition \ref{prop.DP9contraction} are
more complicated than the corresponding constructions for conic bundles.
Their justification relies on machinery presented in an
Appendix on birational contractions.

The proof, given in the following sections, has several ingredients:
\begin{itemize}
\item the forwards construction, which produces a
mildly degenerating Brauer-Severi surface bundle
out of a smooth of $\PP^2$-fibration over the root stack, is carried out
using explicit birational contractions (Section \ref{sec.forwards});
\item the reverse construction, starting with a
mildly degenerating Brauer-Severi surface bundle,
is carried out with similiar techniques (Section \ref{sec.reverse});
\item the verification that every
Brauer-Severi surface bundle $\pi\colon X\to S$, such that $X$ is regular and
precisely the fibers over a regular divisor on $S$ are singular, has
mild degeneration, based on a combination of the
forwards construction and purity for the Brauer group (Section \ref{sec.conc});
\end{itemize}

\section{Forwards construction}
\label{sec.forwards}
In this section we carry out the forwards construction
of Proposition \ref{prop.DP9contraction}.

Let $P\to \sqrt[3]{(S,D)}$ be a smooth $\PP^2$-fibration, such that the
$\mu_3$-action on $P_s$ for every $s\in D$ is balanced.
As in the proof of Proposition \ref{prop.coniccontraction}, there is an affine \'etale
neighborhood $\Spec(B)\to S$ of a point $s\in D$,
where the pre-image of $D$ is principal, say,
defined by the vanishing of $f\in B$, and setting
\[
B':=B[t]/(t^3-f),\qquad V:=[\Spec(B'[u,v,w])/\mu_3],\qquad \widehat{P}:=\PP(V),
\]
we have an isomorphism
\begin{equation}
\label{eqn.isoPV}
\Spec(B)\times_SP\cong \widehat{P}.
\end{equation}
Here, $\mu_3$ acts on by scalar multiplication on
$t$ and $u$, by scalar multiplication squared on $v$, and trivially on $w$.
In particular, $P_\mu$ is a degree $3$ \'etale representable cover
of the gerbe of the root stack, and after blowing up
$P_c$ is smooth over the gerbe of the root stack, consisting of
a line in every $\PP^2$ fiber of the exceptional divisor, and
$\widetilde{P}_\lambda$ is also smooth over the gerbe of the root stack,
consisting geometrically of three pairwise disjoint $(-2)$-curves in every
component of fibers over the gerbe of the root stack,
isomorphic to the minimal resolution of the singular cubic surface
of type $3\mathsf A_2$ \cite{brucewall}.

At the successive blow-ups we have divisors at our disposal as described
below; for the description of combinations that are relatively ample over $S$
we use the fact about ampleness mentioned in \S \ref{ss.cms} in combination with
the equivalence \cite[2.6.2]{EGAIII} of ampleness of a line bundle on a projective scheme with
ampleness of its restriction to every irreducible component, to
characterize relatively ample line bundles.
We adopt the convention that the same notation will be used for
a divisor or line bundle and its pullbacks under the morphisms described below.
But in this section, tilde (\,$\tilde{\ }$\,), as in $\widetilde{E}_1$, always
indicates proper transform.
We let  $\omega^\vee$ denote 
the relative anticanonical bundle
of $P$ over $\sqrt[3]{(S,D)}$.
\begin{itemize}
\item We analyze the first blow-up 
$$
B\ell_{P_\mu}P\ra P.
$$ 
In addition to $\omega^\vee$, we have the
exceptional divisor $E_1$; on geometric fibers over the gerbe of the root stack, 
$E_1$ consists of three components $E_u$, $E_v$, $E_w$, each isomorphic to $\PP^2$, attached along three pairwise disjoint
exceptional divisors $d_u$, $d_v$, $d_w$ of a component $\Sigma_1$ isomorphic to a
degree $6$ del Pezzo surface (DP6),
with the remaining exceptional divisors $\lambda_{uv},\lambda_{vw}, \lambda_{uw}\subset \Sigma_1$
contained in the proper transform of  $P_\lambda$. 
The intersection numbers are:
\[
\begin{array}{c|cc}
&d_u& \lambda_{vw} \\ \hline
\omega^\vee&0&3\\
E_1&-1&2
\end{array}
\]
The projection formula supplies the first row. 
Entries in the second row are intersection numbers on DP6 with $d_u+d_v+d_w$.

Multiples of $\alpha\omega^\vee-E_1$ for $\alpha\in \Q$, $\alpha>2/3$ are relatively ample.
\item
We analyze the second blowup 
$$
B\ell_{P_c}(B\ell_{P_\mu}P)\ra B\ell_{P_\mu}P.
$$  
Let $E_2$ be its exceptional divisor. 
With notation as above 
and components $\ell_u\subset E_u$, etc., of the stabilizer locus $P_c$,
we have
the exact sequence of normal bundles
\[
0\to N_{P_c/E_1}|_{\ell_u}\to N_{P_c/B\ell_{P_\mu}P}|_{\ell_u} \to N_{E_1/B\ell_{P_\mu}P}|_{\ell_u}\to 0
\]
with 
$$
\deg(N_{\ell_u/E_u})=1\quad \text{  and } \quad \deg(N_{E_u/B\ell_{P_\mu}P}|_{\ell_u})=-1
$$
(and similarly for $\ell_v$ and $\ell_w$).
On geometric fibers over the gerbe of the root stack, 
$E_2$ consists of three components  $F_u$, $F_v$, $F_w$, each
isomorphic to a Hirzebruch surface $\F_2$,
glued along $(-2)$-curves $c_u$, respectively, $c_v$, $c_w$,  to $\widetilde{E}_u$, respectively, $\widetilde{E}_v$, $\widetilde{E}_w$,
and along a fiber $f_{2,u}$  (respectively, $f_{2,v}$, $f_{2,w}$) to $\Sigma_2$,
the blow-up of $\Sigma_1$ at three points; note that $\Sigma_2$ is isomorphic to the minimal resolution of the singular
cubic surface of type $3\mathsf A_2$.
The class of the exceptional divisor $E_2$, restricted to $F_u$, etc., is 
$\cO_{\PP(\cO_{\PP^1}(1)\oplus \cO_{\PP^1}(-1))}(-1)$.
The further restriction to $c_u$, etc.,  is
$\cO_{\PP(\cO_{\PP^1}(1)\oplus 0)}(-1)\cong \cO_{\PP^1}(1)$
and to $f_{2,u}$, etc.,  is isomorphic to $\cO_{\PP^1}(-1)$;
these explain the last two entries in the third row:
\[
\begin{array}{c|cccc}
&\tilde{d}_u&\tilde\lambda_{vw}& c_u & f_{2,u}\\ \hline
\omega^\vee&0&3&0 & 0\\
E_1&-1&2& -1& 0\\
E_2&1&1&1 & -1
\end{array}
\]
Intersection numbers on $\Sigma_2$
give
the first two entries in the third row.

Multiples of $\beta(\alpha \omega^\vee-E_1)-E_2$ with 
$$
\alpha>(1/3)(2+1/\beta) \qquad \text{and} \qquad \beta>1
$$ 
are relatively ample.
\item 
We analyze the third blowup 
$$
B\ell_{\widetilde{P}_\lambda}(B\ell_{P_c}(B\ell_{P_\mu}P))\ra B\ell_{P_c}(B\ell_{P_\mu}P). 
$$
Let $E_3$ be its exceptional divisor.
On geometric fibers over the gerbe of the root stack 
$E_3$ consists of components $Q_{uv}$, $Q_{vw}$, $Q_{uw}$,
fibered over $\tilde \lambda_{uv}$, etc.
With an exact sequence of normal bundles, as in the second blowup, we see that
each component is isomorphic to $\PP^1\times \PP^1$, with
$\cO(E_3)$ of degree $-2$ along ruling sections and
$-1$ along fibers, and is glued along a fiber to the
exceptional divisor $e_u$ of a blow-up of $\widetilde{E}_u$, etc., along another fiber to the
exceptional divisor $f_u$ of a blow-up of $F_u$, etc.,
and along a ruling section (in the table, $\tilde \lambda_{vw}$, etc.)\ to $\Sigma_3\cong \Sigma_2$.
\[
\begin{array}{c|cccccc}
&\tilde d_u&e_u&\tilde\lambda_{vw}&c_u&\tilde f_{2,u}& f_u\\ \hline
\omega^\vee&0&0&3&0&0&0\\
E_1&-1&0&2&-1&0&0\\
E_2&1&0&1&1&-1&0\\
E_3&1&-1&-2&0&1&-1
\end{array}
\]

Multiples of $\gamma(\beta(\alpha \omega^\vee-E_1)-E_2)-E_3$ with
\[
\alpha>\frac{2}{3}+\frac{1}{3\beta}-\frac{2}{3\beta\gamma},\qquad
\beta>1+\frac{1}{\gamma},\qquad\text{and}\qquad
\gamma>1
\]
are relatively ample.
\end{itemize}

The degree $8$ components of geometric fibers over the root stack
$\widetilde{E}_u$, $\widetilde{E}_v$, $\widetilde{E}_w$, $Q_{uv}$, $Q_{vw}$, $Q_{uw}$
comprise the divisor
\[ \widetilde{E}_1\cup E_3 \]
in the class of $E_1-E_2+E_3$, with evident contraction
$\psi$ to a stack $Z$, isomorphic to each of the
intersections
$\widetilde{E}_1\cap E_2$, $\widetilde{E}_1\cap E_3$, $E_2\cap E_3$.
Proposition \ref{prop.semiampleoverStoscheme} is applicable to the contraction
$\psi\colon \widetilde{E}_1\cup E_3\to Z$
with the relatively ample line bundle associated with (an integral multiple of)
$\gamma(\beta(\alpha \omega^\vee-E_1)-E_2)-E_3$ for
\[ \alpha\beta>1,\qquad
(1-\alpha)\beta<1,\qquad
(2\alpha-1)\beta<1,\qquad \gamma=\frac{1}{(1-\alpha)\beta},
\]
and value of $m$, coefficient appearing in the twist of the
relatively ample line bundle by $m(E_1-E_2+E_3)$,
equal to (the same integral multiple of) $(\alpha\beta-1)/(1-\alpha)\beta$.
We obtain a contraction
\[ \nu\colon B\ell_{\widetilde{P}_\lambda}(B\ell_{P_c}(B\ell_{P_\mu}P))\to P_1, \]
restricting to $\psi\colon \widetilde{E}_1\cup E_3\to Z\subset P_1$, with
$\cI_Z/\cI_Z^2$ locally free of rank $3$.

By Remark \ref{rem.semiampleoverStoschemeii},
the effective Cartier divisor $E_1+E_3$, viewed as invertible sheaf with global section,
pushes forward under $\nu$ to a Cartier divisor on $P_1$ that we denote by
$\Psi$.
Over $\Spec(B)$,
where we have an isomorphism \eqref{eqn.isoPV} and therefore
individually defined divisors $\widetilde{E}_u$, etc.,
further Cartier divisors arise by pushing forward
\[ \widehat{G}_u:=2\widetilde{E}_u+\widetilde{E}_v+3\widetilde{F}_u+2Q_{uv}+Q_{uw} \]
and the analogously defined
$\widehat{G}_v$ and $\widehat{G}_w$ to
\[ \widehat{P}_1:=\Spec(B)\times_SP_1. \]
So the union of the blown up $\F_2$ components of $\widehat{P}_1$
over the gerbe of the root stack is a Cartier divisor,
as is $3$ times each individual blown up $\F_2$ component.
Since $\nu$ is a birational contraction of a normal scheme,
$P_1$ is normal.

We claim, $P_1$ has $\mathsf A_2$-singularities along $Z$.
This is a local assertion, so we may work over $\widehat{P}_1$, which is a quotient stack
by $\mu_3$.
We use the language of schemes (with $\mu_3$-action) in the following argument.
Fix a closed point $\hat r\in \widehat{Z}$
over $s'\in \Spec(B')$, corresponding to some $\mathfrak{m}'\subset B'$, over a given $s\in D$.
Let $n=\dim(B'_{\mathfrak{m}'})$, and let $t$, $g_1$, $\dots$, $g_{n-1}$
be a regular system of parameters, i.e., elements that form a basis of
$\mathfrak{m}'/\mathfrak{m}'^2$.
If $\hat{\mathfrak{n}}$ denotes the maximal ideal of the local ring $\cO_{\widehat{P}_1,\hat r}$,
where the residue field of $\hat r$ is a finite extension $\ell$ of
the residue field $k$ of $s'$,
then since $\widehat{Z}$ is regular there is the exact sequence of
vector spaces over $\ell$ \cite[16.9.13]{EGAIV}
\begin{equation}
\label{eqn.learnsequence}
0\to \cI_{\widehat{Z}}/\hat{\mathfrak{n}}\cI_{\widehat{Z}} \to \hat{\mathfrak{n}}/\hat{\mathfrak{n}}^2\to
\hat{\mathfrak{n}}/(\cI_{\widehat{Z}}+\hat{\mathfrak{n}}^2)\to 0,
\end{equation}
from which we learn
\[ \dim(\hat{\mathfrak{n}}/\hat{\mathfrak{n}}^2)=n+3. \]
As well, by Proposition \ref{prop.semiampleoverStoscheme},
we have the compatibility of the contraction $\hat\nu$ with base change, e.g.,
via $\{s'\}\to \Spec(B')$, and this way we learn
that the fiber of $\widehat{P}_1$ over $s'$ consists of a singular cubic surface of type
$3\mathsf A_2$ joined to three components, each isomorphic to the
blow-up of $\F_2$ at one point.
In particular,
\[
\dim(\hat{\mathfrak{n}}/(\mathfrak{m}'\cO_{\widehat{P}_1,\hat r}+\hat{\mathfrak{n}}^2))=
\begin{cases}
4,&\text{if $3$ components meet at $\hat r$}, \\
3,&\text{otherwise}.
\end{cases}
\]
In both cases, we let $x$ and $y$ be local defining equations of the
Cartier divisors on $\widehat{P}_1$, each consisting (on fibers) of one blown up $\F_2$ component with multiplicity $3$,
and $z$, of their union with multiplicity $1$;
in the case $3$ components meet at $\hat r$
we additionally let $q$ be a local defining equation for the
singular cubic surface divisor component.

We have, after adjusting by units,
\[ xy=z^3 \]
and, in the two cases,
\[ t=qz\in \hat{\mathfrak{n}}^2,\qquad\text{respectively,}\qquad t=z. \]
The kernel
\[ (\mathfrak{m}'\cO_{\widehat{P}_1,\hat r}+\hat{\mathfrak{n}}^2)/\hat{\mathfrak{n}}^2 \]
of the canonical homomorphism
$\hat{\mathfrak{n}}/\hat{\mathfrak{n}}^2\to
\hat{\mathfrak{n}}/(\mathfrak{m}'\cO_{\widehat{P}_1,\hat r}+\hat{\mathfrak{n}}^2)$
is generated by $\mathfrak{m}'$.
So the kernel is,
in the two cases, of dimension $n-1$ with basis $g_1$, $\dots$, $g_{n-1}$, respectively,
of dimension $n$ with basis $t$, $g_1$, $\dots$, $g_{n-1}$.

Suppose, first, $\hat r$ is the point where three components come together.
We know, modulo $q$, the point $\hat r$ is an $\mathsf A_2$-singular point
of the cubic surface component, with
$\hat{\mathfrak{n}}/(\mathfrak{m}'\cO_{\widehat{P}_1,\hat r}+q\cO_{\widehat{P}_1,\hat{r}}+\hat{\mathfrak{n}}^2)$
of dimension $3$, spanned by $x$, $y$, and $z$.
So $q$, $x$, $y$, and $z$ span $\hat{\mathfrak{n}}/(\mathfrak{m}'\cO_{\widehat{P}_1,\hat r}+\hat{\mathfrak{n}}^2)$;
they define a morphism to $\A^4_{B'}$ which is unramified at $\hat r$ and hence by \cite[18.4.7]{EGAIV}
restricts and lifts to a closed immersion of a
neighborhood $\Spec(C)$ of $\hat r$ in an \'etale affine neighborhood $\Spec(\Delta)$
of the point $0$ over $s'$ in $\A^4_{B'}$;
let $\delta_q$, $\delta_x$, $\delta_y$, $\delta_z\in \Delta$ denote the coordinate functions from $\A^4_{B'}$.
Then, setting $\widetilde{\Delta}:=\Delta/(t-\delta_q\delta_z)$, we see that
$\Delta\to C$ factors through
\begin{equation}
\label{eqn.deltaprime}
\widetilde{\Delta}/(\delta_x\delta_y-\delta_z^3).
\end{equation}
The element $\delta_x\delta_y-\delta_z^3$ is irreducible in the local ring of
$\widetilde{\Delta}$ at the point $0$ over $s'$, so the localization of \eqref{eqn.deltaprime}
is an integral domain, and we have found the desired local form of the
singularity at $\hat r$.

Now suppose $\hat r$ is a point of the component $\PP^1$
where two blown up $\F_2$ components meet.
The Cartier divisor defined by $x$,
restricted to the other blown up $\F_2$ component,
defines $\PP^1$ with multiplicity $1$,
and the same is true with $y$ and the first blown up
$\F_2$ component.
It follows that $x$, $y$, and a polynomial of degree $[\ell:k]$ in
a local parameter $p$ for $\PP^1$
span $\hat{\mathfrak{n}}/(\mathfrak{m}'\cO_{\widehat{P}_1,\hat r}+\hat{\mathfrak{n}}^2)$.
We argue as above, with morphism to $\A^3_{B'}$ determined by $x$, $y$, and $p$
and $\Delta\to C$ factoring through $\Delta/(\delta_x\delta_y-t^3)$.

On $P_1$ we have Cartier divisors $\Psi\cup \nu(\Sigma_3)$ and $\Psi$.
Hence $\nu(\Sigma_3)$ is a Cartier divisor, which becomes linearly equivalent
to $-\Psi$ after base change to any affine \'etale chart of
$\sqrt[3]{(S,D)}$ where the pre-image of the gerbe of the root stack is
principal.
We apply Proposition \ref{prop.semiampleoverStoscheme} with
relatively ample $\cO_{P_1}(1)$ and $m$ taken so that
the twist by $m\,\nu(\Sigma_3)$ has degree $0$ on $\tilde f_{2,u}$.
Thus we obtain a contraction
\[ \rho\colon P_1\to P_2, \]
where $P_2$ is a Brauer-Severi surface bundle over
$\sqrt[3]{(S,D)}$.
Under $\rho$ we have
$\nu(\Sigma_3)$ contracting to a copy $G$
in $P_2$ of the gerbe of the root stack,
with $\mathcal{I}_G/\mathcal{I}_G^2$ locally free of rank $4$.
An exact sequence analogous to \eqref{eqn.learnsequence}
shows that $\dim(\hat{\mathfrak{n}}/\hat{\mathfrak{n}}^2)=n+3$, where $\hat{\mathfrak{n}}$ is the
maximal ideal of the local ring, after base change to an affine \'etale chart
of $\sqrt[3]{(S,D)}$, at a closed point $\hat r$ over $G$; in particular, $P_2$ has
hypersurface singularities.

As above, we may work over $\Spec(B)$ and
perform previous analysis to show that $3$ times each $\F_1$ component over
the gerbe of the root stack is a Cartier divisor.
This leads to an \'etale local defining equation at a point $\hat r$ as above
of $\delta_x\delta_y\delta_z-t^3$, with analogous notation to that appearing in
\eqref{eqn.deltaprime}.
With notation as before, we have
an affine chart $\Spec(C)$ of
\[ \widehat{P}_2:=\Spec(B)\times_SP_2, \]
which we may without loss of
generality take to be $\mu_3$-invariant, and a closed immersion
\[ \Spec(C)\to \Spec(\Delta) \]
identifying the local ring $C_{\hat{\mathfrak{n}}}$ with
$\Delta_{\hat{\mathfrak{n}}'}/(\delta_x\delta_y\delta_z-t^3)$,
where $\hat{\mathfrak{n}}'$ denotes the maximal ideal, corresponding to the image
of $\hat r$ in $\Spec(\Delta)$.
Although $\Delta$ is not determined canonically, its strict henselization
at $\hat{\mathfrak{n}}'$ is canonically determined (up to
fixing a separable closure of the residue field): it is a
strict henselization of affine $3$-space over $B'$.

By Proposition \ref{prop.coarsemodulispace}, there exists
a scheme $P_3$, Fano over $S$, with
\[ \sqrt[3]{(S,D)}\times_SP_3\cong P_2. \]
By Remark \ref{rem.cms}, $P_2\to P_3$ is a coarse moduli space, so with
the above notation, $P_3$ has affine coordinate ring $C^{\mu_3}$.
We claim, $P_3$ is regular; equivalently,
\[ \widehat{P}_3:=\Spec(B)\times_SP_3 \]
is regular.
It suffices to verify the claim after passing to a strict henselization
$(C^{\mu_3})^{sh}$ of $P_3$.
By \cite[18.8.10]{EGAIV}, we have
\[ C\otimes_{C^{\mu_3}}(C^{\mu_3})^{sh}\cong C^{sh}. \]
Then a strict Henselization of $P_2$
(for Henselization of a stack, see \cite[\S 2.7]{AHR}) takes the form
\[ [\Spec(C)/\mu_3]\times_{\Spec(C^{\mu_3})}\Spec((C^{\mu_3})^{sh})\cong
[\Spec(C^{sh})/\mu_3]. \]
We have, $\mu_3$-equivariantly,
\[ C^{sh}\cong \Delta^{sh}/(\delta_x\delta_y\delta_z-t^3) \]
at a point where three components meet.
Hence upon taking $\mu_3$-invariants,
\[ (C^{\mu_3})^{sh}\cong (C^{sh})^{\mu_3}
\cong (\Delta^{sh})^{\mu_3}/(\delta_x\delta_y\delta_z-f). \]
The ring $(\Delta^{sh})^{\mu_3}$ is a strict Henselization of affine
$3$-space over $B$ and hence is regular.
A similar analysis takes care of the regularity at a point
where two components meet.
So, $P_3$ is regular.

\section{Reverse construction}
\label{sec.reverse}
In this section, we start with a mildly degenerating
Brauer-Severi surface bundle and carry out the reverse construction of
Proposition \ref{prop.DP9contraction}.

For every $s\in D$, the fiber $X_s$ has,
geometrically, three irreducible components.
Lemma \ref{lem.pi0} determines a degree $3$ finite \'etale cover
$$
\widetilde{D}:=\pi_0(X^{\mathrm{sm}}/S)\times_SD\to D.
$$
For $s\in D$, let $\Spec(B)\to S$ be an \'etale neighborhood of $s$ that
trivializes $\widetilde{D}\to D$.
By Remark \ref{rem.pi0},
the three sections of 
\[
\Spec(B)\times_S\widetilde{D}\to \Spec(B)\times_SD
\] 
determine three
divisors $W_1$, $W_2$, $W_3$ on $\Spec(B)\times_SX$, whose intersections with
the fiber $X_s$ of any $s\in \Spec(B)\times_SD$ are the irreducible components of
$X_s$.
Let $w_1$, $w_2$, $w_3$ be respective local defining equations of the divisors
near a singular point $r\in X_s$.
Then, after modification by a unit, we have
$f=w_1w_2w_3$.

We claim, there is an \'etale local equation
\[ 
w_1w_2w_3-f,\qquad  \text{respectively,}\qquad  w_1w_2-f,
\]
(after suitable renumbering of indices)
for $\Spec(B)\times_SX$ in $\A^3_B$, according to whether $x$ is a point of $X_s$
where $3$ or $2$ components meet.

To establish the claim, we argue as in Section \ref{sec.forwards}.
Let $\mathfrak{m}\subset B$ be the maximal ideal corresponding to $s$ and
$n=\dim(B_{\mathfrak{m}})$.
We may choose a regular system of parameters for $B_{\mathfrak{m}}$ of
the form $f$, $g_1$, $\dots$, $g_{n-1}$.
The space $\mathfrak{n}/\mathfrak{n}^2$ has dimension $n+2$.
Since the fiber $X_s$ has hypersurface singularities, at a singular
point $r\in X_s$ with corresponding maximal ideal
$\mathfrak{n}\subset \cO_{X,r}$ we have
\[ \dim(\mathfrak{n}/(\mathfrak{m}\cO_{X,r}+\mathfrak{n}^2))=3. \]
The claim follows, as in Section \ref{sec.forwards}, once we show that
$\mathfrak{n}/(\mathfrak{m}\cO_{X,r}+\mathfrak{n}^2)$ is spanned by
$w_1$, $w_2$, $w_3$, respectively by two of them and a local parameter
for the component $\PP^1$ of the singular locus of the fiber.

We have
\[ \dim((\mathfrak{m}\cO_{X,r}+\mathfrak{n}^2)/\mathfrak{n}^2)=n-1. \]
The image of $f$ in $\cO_{X,r}$ lies in $\mathfrak{n}^2$.
Indeed, if not, then $\cO_{X,r}/(f)$ would be regular, hence also the further
localization, where we pass to the residue field at the generic point of $D$, would be regular.
But the relative singular locus of $X\times_SD$ over $D$ is flat
(since it has constant Hilbert polynomial), hence some point of the
singular locus over the generic point of $D$ specializes to $r$, contradiction.
So, the images of $g_1$, $\dots$, $g_{n-1}$ form a basis of
$(\mathfrak{m}\cO_{X,r}+\mathfrak{n}^2)/\mathfrak{n}^2$.
Now, modulo $g_1$, $\dots$, $g_{n-1}$, we still have $f=w_1w_2w_3$,
where $f=0$ now defines the fiber $X_s$, and hence $w_1$, $w_2$, $w_3$
define its three irreducible components, each with multiplicity one.

Suppose, first, that $r$ is the point where all three components meet.
Then, for all $\{i,j,k\}=\{1,2,3\}$ the classes of $w_j$ and $w_k$ are a basis of
$\mathfrak{n}/(\mathfrak{m}\cO_{X,r}+w_i\cO_{X,r}+\mathfrak{n}^2)$.
So, the classes of $w_1$, $w_2$, $w_3$ are a basis of
$\mathfrak{n}/(\mathfrak{m}\cO_{X,r}+\mathfrak{n}^2)$,
and we are done in this case.

If $r$ is a point where two components meet, say,
those defined by $w_1$ and $w_2$,
then for $\{i,j\}=\{1,2\}$ and local
parameter $p$ for the component $\PP^1$ of the singular locus,
$w_j$ and $p$ are a basis of
$\mathfrak{n}/(\mathfrak{m}\cO_{X,r}+w_i\cO_{X,r}+\mathfrak{n}^2)$.
It follows that $w_1$, $w_2$, $p$ span
$\mathfrak{n}/(\mathfrak{m}\cO_{X,r}+\mathfrak{n}^2)$,
and we are done in this case as well.

The claim implies the corresponding singularity type of
$\sqrt[3]{(S,D)}\times_SX$, where we have $f=t^3$ in the coordinate ring of
an affine \'etale chart.
The two blowing up steps therefore supply a resolution of singularities of
$\sqrt[3]{(S,D)}\times_SX$.
The three contractions then may be performed, as in Section \ref{sec.forwards},
noticing by the normal cone description of Proposition \ref{prop.semiampleoverStoscheme}
that the outcome of each contraction is regular.

\section{Regular degeneracy locus implies mild degeneration}
\label{sec.conc}
In this section we complete the proof of Proposition \ref{prop.DP9contraction} by
showing that a Brauer-Severi surface bundle $\pi\colon X\to S$ with
$X$ regular and singular fibers precisely along a regular divisor $D\subset S$,
always has mild degeneration.
We may assume that $S$ is quasi-compact and integral.
We consider various cases, depending on the dimension of $S$.

If $\dim(S)=1$ then $\dim(X)=3$, hence if $s\in S$ is any closed point
then $3$ is an upper bound on the
embedding dimension of $X_s$ at any of its closed points.
So $X\to S$ has mild degeneration.

If $\dim(S)=2$ then by the previous case we know that the fiber at
the generic point of every component of $D$ is, geometrically, a union
of three Hirzebruch surfaces.
So, there is an open subscheme $S'$ of $S$ such that
$S'\times_SX\to S'$ has mild degeneration and
$\dim(S\smallsetminus S')=0$.

We apply the reverse construction (Section \ref{sec.reverse}) to
$S'\times_SX\to S'$ to obtain a smooth $\PP^2$-bundle
\[ P'\to S'\times_S\sqrt[3]{(S,D)}. \]
By Proposition \ref{prop.Brbasic} (iii), this extends to a
$\PP^2$-bundle
\[ P\to \sqrt[3]{(S,D)}. \]
We now apply the forwards construction to obtain a
mildly degenerating Brauer-Severi surface bundle
$\hat\pi\colon \widehat{X}\to S$.

By Remark \ref{rem.10sections} and the fact that a line bundle,
ample on fibers, is relatively ample, we have the ampleness of the
dual of the relative dualizing sheaf of $\widehat{X}\to S$.
This holds as well for $X\to S$, where the ampleness on geometric fibers
over $S\smallsetminus S'$ is dealt with by
\cite[4.5.14]{EGAII}, which tells us
that ampleness of a line bundle on a Noetherian scheme is independent of any
non-reduced structure, in combination with the fact that a
cone over a twisted cubic curve has Picard group isomorphic to $\Z$.

Since
\[ X\cong \Proj\big(\bigoplus_{n\ge 0} \pi_*(\omega^\vee_{X/S})^n\big)\qquad\text{and}\qquad
\widehat{X}\cong \Proj\big(\bigoplus_{n\ge 0} \hat\pi_*(\omega^\vee_{\widehat{X}/S})^n\big), \]
where the direct image sheaves are locally free for all $n\gg 0$ with isomorphic
restrictions over $S'$, they are isomorphic for all $n\gg 0$, hence $X\cong \widehat{X}$
is a Brauer-Severi surface bundle with mild degeneration.

We observe that the argument for $\dim(S)=2$ does not use the regularity of $X$,
but only the flatness of $X\to S$ and the regularity of $S'\times_SX$.

The case $\dim(S)\ge 3$ reduces to the case $\dim(S)=2$ by slicing the base.
Let $s\in S$ be a point where the geometric fiber is irreducible with
reduced subscheme a cone over a twisted cubic curve.
Let $f$, $g_1$, $\dots$, $g_{d-1}$ be a regular system of parameters for
$\cO_{S,s}$, and let $T$ be the regular two-dimensional scheme
$\Spec(\cO_{S,s}/(g_2,\dots,g_{d-1}))$.
Let $Y:=T\times_SX$.
We claim, $Y\smallsetminus Y_s$ is regular.
This is clear at points in the relative smooth locus of $Y\smallsetminus Y_s\to T\smallsetminus \{s\}$.
Consider the generic point $\eta$ of $T\cap D$ with corresponding maximal ideal $\mathfrak{m}\subset \cO_{S,\eta}$,
and a closed point $r\in (X_\eta)^{\mathrm{sing}}$,
with corresponding maximal ideal $\mathfrak{n}\subset \cO_{X,r}$.
We use the exact sequence
\[
0\to (\mathfrak{m}\cO_{X,r}+\mathfrak{n}^2)/\mathfrak{n}^2\to
\mathfrak{n}/\mathfrak{n}^2\to
\mathfrak{n}/(\mathfrak{m}\cO_{X,r}+\mathfrak{n}^2)\to 0.
\]
The space in the middle has dimension $d+1$, on the right, $3$, hence
on the left, $d-2$.
The space $\mathfrak{m}/\mathfrak{m}^2$ has dimension $d-1$ and is
spanned by $f$, $g_2$, $\dots$, $g_{d-1}$.
Moreover, $f$ lies in the kernel of
\[ \mathfrak{m}/\mathfrak{m}^2\to (\mathfrak{m}\cO_{X,r}+\mathfrak{n}^2)/\mathfrak{n}^2. \]
So $g_2$, $\dots$, $g_{d-1}$ are linearly independent in $\mathfrak{n}/\mathfrak{n}^2$,
and hence $\cO_{Y,r}$ is regular.

By the observation made after the argument for the case $\dim(S)=2$,
we may conclude that $Y\to T$ has mild degeneration, and this is a contradiction.

\section{Local analysis I}
\label{sec.local1}
For the next result we use the notation of Section \ref{sec.forwards}
but work with schemes with $\mu_3$-action: $P'$ will denote
$\Proj(B'[u,v,w])$, so $\widehat{P}=[P'/\mu_3]$, etc.

\begin{prop}
\label{prop.10sections}
With the notation as remarked above,
$\omega^\vee_{P'_2/\Spec(B')}$
is very ample over $\Spec(B')$,
and formation of its global sections commutes with
base change to an arbitrary Noetherian scheme over $\Spec(B')$.
Identifying the pullback of $\omega^\vee_{P'_2/\Spec(B')}$ to
$B\ell_{\widetilde{P}'_\lambda}(B\ell_{P'_c}(B\ell_{P'_\mu}P'))$
with 
$$
\omega^\vee-2E_1-E_2-E_3
$$ 
and
sections with cubic forms in $u$, $v$, $w$
vanishing appropriately along the $E_i$,
the sections form a free $B'$-module with basis
\[ t^3w^3,\ \ t^2uw^2,\ \ tvw^2,\ \ tu^2w,\ \ uvw,\ \ t^2v^2w,
\ \ t^3u^3,\ \ t^2u^2v,\ \ tuv^2,\ \ t^3v^3. \]
\end{prop}

The identification of line bundles in
Proposition \ref{prop.10sections} comes from the
observation that the line bundle associated with
$\omega^\vee-2E_1-E_2-E_3$ is the dual of the
relative dualizing sheaf of
$B\ell_{\widetilde{P}'_\lambda}(B\ell_{P'_c}(B\ell_{P'_\mu}P'))$ over $\Spec(B')$
(by standard behavior of dualizing sheaf under blowing up) and
under the morphism to $P'_2$ has direct image $\omega^\vee_{P'_2/\Spec(B')}$.

\begin{proof}
The observation from Remark \ref{rem.10sections}, that the line bundle
$\omega^\vee_{P'_2/\Spec(B')}$ on fibers is very ample with
space of global sections of dimension $10$, implies the
assertions in the first statement.

For the remaining assertion, we compute with local charts.
We work, first, over the open subscheme 
\[
W':=\Spec(B'[u,v])
\]
of $P'$ and
use notation with the letter $P$ replaced by $W$ to denote restriction
over $W'$. The blowup
$B\ell_{W'_\mu}W'$ is covered by three charts:

Chart $1$: coordinates $u$, $t_1$, $v_1$ with $t=ut_1$ and $v=uv_1$.
Equation for $E_1$: $u=0$; equation for $W'_c$: $u=v_1=0$.

Chart $2$: coordinates $v$, $t_2$, $u_2$ with $t=vt_2$ and $u=vu_2$.
Equation for $E_1$: $v=0$; trivial intersection with $W'_c$,
equation for $\widetilde{W}'_\lambda$: $t_2=u_2=0$.

Chart $3$: coordinates $t$, $u_3$, $v_3$ with $u=tu_3$ and $v=tv_3$.
Equation for $E_1$: $t=0$; equation for $W'_c$: $t=v_3=0$.

\medskip

The next blowup $B\ell_{W'_c}(B\ell_{W'_\mu}W')$ is covered by five charts, Chart $2$ and:

Chart $1'$: coordinates $u$, $t_1$, $v'_1$ with $v_1=uv'_1$.
Equation for $E_2$: $u=0$;
equation for $\widetilde{W}'_\lambda$: $t_1=v'_1=0$.

Chart $1''$: coordinates $t_1$, $v_1$, $u''$ with $u=v_1u''$.
Equation for $E_2$: $v_1=0$; trivial intersection with $\widetilde{W}'_\lambda$.

Chart $3'$: (analogous)

Chart $3''$: (analogous)
\medskip

The next blowup
$B\ell_{\widetilde{W}'_\lambda}(B\ell_{W'_c}(B\ell_{W'_\mu}W'))$ is covered
by seven charts,
of which we need just two:

Chart $1'_{\mathrm{a}}$: coordinates $u$, $t_1$, $v'_{1,\mathrm{a}}$
with $v'_1=t_1v'_{1,\mathrm{a}}$.
Equation for $E_3$: $t_1=0$.

Chart $2_{\mathrm{a}}$: coordinates $v$, $t_2$, $u_{2,\mathrm{a}}$
with $u_2=t_2u_{2,\mathrm{a}}$.
Equation for $E_3$: $t_2=0$

The union of charts $1'_{\mathrm{a}}$ and $2_{\mathrm{a}}$ contains
the generic point of every component of $E_1$, $E_2$, and $E_3$ over $W'$.
With cyclic shifts of variables $u$, $v$, $w$ we have additional
affine opens $U'$ and $V'$ of $P'$ with analogous charts of the
blow-ups.
By considering the analogous charts, too, we find that sections of 
$\omega^\vee-2E_1-E_2-E_3$ are cubic forms in $u$, $v$, $w$ with
coefficients in $B'$, such that
after applying any cyclic permutation of $u$, $v$, $w$ and setting $w$ to $1$
the image in the coordinate ring of Chart $1'_{\mathrm{a}}$
lies in the ideal
\[
u^3t_1B[u,t_1,v'_{1,\mathrm{a}}]/(u^3t_1^3-f),
\]
and the image in the coordinate ring of Chart $2_{\mathrm{a}}$ lies in
the ideal
\[
v^2t_2B[v,t_2,u_{2,\mathrm{a}}]/(v^3t_2^3-f).
\]

We compute with respect to the basis of degree $3$ monomials in $u$, $v$, $w$.
The kernel of
\[ B'^{10}\to (B/f)[u,t_1,v'_{1,\mathrm{a}}]/(u^3t_1) \]
may be computed with the help of the observation that the map factors through
$((B/f)[t]/(t^3))^{10}$, which along with
$(B/f)[u,t_1,v'_{1,\mathrm{a}}]/(u^3t_1)$ is a free $B/f$-module.
The kernel of
\[ B'^{10}\to (B/f)[v,t_2,u_{2,\mathrm{a}}]/(v^2t_2) \]
may be computed similarly.
These computations lead to the determination of the module of such forms
as the free module with the claimed basis.
\end{proof}

\section{Local analysis II}
\label{sec.local2}

We exhibit a Brauer-Severi surface fibration over a two-dimensional base,
with singular fibers along a union of two intersecting divisors and
smooth total space.
For conic bundles this is easy:
\[ \Proj(k[s,t,x,y,z]/(sx^2+ty^2+z^2)). \]
This conic bundle over $\A^2=\Spec(k[s,t])$ has smooth total space and
singular fibers along the divisor defined by $st=0$.

For Brauer-Severi surface bundles this is more complicated.
Over the generic point, the fiber is an
anticanonically embedded Brauer-Severi surface,
defined in $\PP^9$ by
\[ \frac{10\cdot 11}{2}-h^0(\PP^2,(\omega^\vee_{\PP^2})^2)=55-28=27 \]
quadratic equations.
We will obtain a equations for a Brauer-Severi surface bundle
over $\A^2$ with smooth total space and
singular fibers along the union of coordinate axes by applying results from
\S\ref{sec.brauerseverisurfacebundles}.
The existence of such a Brauer-Severi surface bundle
has been shown by Maeda \cite{maeda} using different methods
(ideals and orders in a central simple algebra).

Let $k$ be a field of characteristic different from $3$ containing a
primitive cube root of unity $\zeta$, and let $\mu_3\times \mu_3$ act on
$\A^2\times \PP^2$ over $k$, where cube root of unity $\zeta$ in the
first factor acts by
\[ (s,t,u:v:w)\mapsto (\zeta s,t,\zeta u:\zeta^2 v:w) \]
and in the second factor acts by
\[ (s,t,u:v:w)\mapsto (s,\zeta t,w:u:v). \]
Then
\[ [(\A^2\smallsetminus \{0\})\times \PP^2/\mu_3\times \mu_3] \]
is a smooth $\PP^2$-fibration satisfying the hypotheses of
Proposition \ref{prop.DP9contraction}.
There is a corresponding mildly degenerating
Brauer-Severi surface bundle
\begin{equation}
\label{eqn.brauerseveribundleX}
\pi\colon X\to \A^2\smallsetminus \{0\}.
\end{equation}
By Proposition \ref{prop.10sections} applied to the
pullback
\begin{equation}
\label{eqn.brauerseveribundleXprime}
\pi'\colon X'\to \A^2\smallsetminus \{0\}
\end{equation}
of $\pi$ by $(s,t)\mapsto (s^3,t^3)$,
first restricted over $\A^1\times(\A^1\smallsetminus \{0\})$ and then
restricted $(\A^1\smallsetminus \{0\})\times\A^1$,
the space of global sections of $\omega^\vee_{X'/\A^2\smallsetminus \{0\}}$
is the intersection of
\begin{align*}
\mathcal{S}&:=\\
&k[s,t,t^{-1}]\langle s^3w^3,s^2uw^2,svw^2,su^2w,uvw,s^2v^2w,
s^3u^3,s^2u^2v,suv^2,s^3v^3\rangle
\end{align*}
and
\begin{align*}
\mathcal{T}&:=\\
&k[s,s^{-1},t]\langle t^3\tilde w^3,t^2\tilde u\tilde w^2,t\tilde v\tilde w^2,
t\tilde u^2\tilde w,\tilde u\tilde v\tilde w,t^2\tilde v^2\tilde w,
t^3\tilde u^3,t^2\tilde u^2\tilde v,t\tilde u\tilde v^2,t^3\tilde v^3\rangle,
\end{align*}
where
\begin{align*}
\tilde u&=\zeta u+\zeta^2 v+w,\\
\tilde v&=\zeta^2 u+\zeta v+w,\\
\tilde w&=u+v+w.
\end{align*}
Indeed, the action of cyclically permuting the coordinates
$u$, $v$, $w$ is compatible with the diagonal action by
$\zeta$, $\zeta^2$, $1$ on $\tilde u$, $\tilde v$, $\tilde w$.

Only nonnegative powers of $s$ and $t$ show up in an element of
$\mathcal{S}\cap \mathcal{T}$.
For a polynomial $f\in k[s,t,u,v,w]$,
homogeneous of degree $3$ in $u$, $v$, $w$:
\begin{itemize}
\item if $s^3\mid f$, then $f\in \mathcal{S}$;
\item if $t^3\mid f$, then $f\in \mathcal{T}$;
\item generally, if we write 
$$
f=\sum_{a,b\ge 0}s^at^bf_{ab},\quad \text{with}\quad 
f_{ab}\in k[u,v,w],
$$ 
for all $a$ and $b$, then
$f\in \mathcal{S}$ if and only if $s^af_{ab}\in \mathcal{S}$
for all $a$ and $b$, and $f\in \mathcal{T}$ if and only if
$t^bf_{ab}\in \mathcal{T}$ for all $a$ and $b$.
\end{itemize}
We easily characterize the homogeneous $f\in k[u,v,w]$ of degree $3$
with $s^af\in \mathcal{S}$ for $a\le 2$:

\[
\begin{array}{c|ccc}
a&&\text{$k$-basis} \\ \hline
0&uvw\\
1&\text{above and}&u^2w,\ \ vw^2,\ \ uv^2\\
2&&\text{above and}&uw^2,\ \ v^2w,\ \ u^2v
\end{array}
\]

\noindent
Similarly, with a bit of linear algebra we
characterize the homogeneous $f\in k[u,v,w]$ of degree $3$
with $t^bf\in \mathcal{T}$ for $b\le 2$:
\[
\begin{array}{c|cccc}
b&&&\text{$k$-basis} \\ \hline
0&&u^3+v^3+w^3-3uvw\\
1&\text{and}&u^2w+\zeta vw^2+\zeta^2 uv^2,&uw^2+\zeta v^2w+\zeta^2 u^2v,&u^3+\zeta^2 v^3+\zeta w^3\\
2&\text{and}&u^2w+\zeta^2 vw^2+\zeta uv^2,&uw^2+\zeta^2 v^2w+\zeta u^2v,&u^3+\zeta v^3+\zeta^2 w^3
\end{array}
\]
With these data, we determine the space
of $f$ with $s^at^bf\in \mathcal{S}\cap \mathcal{T}$ for
$0\le a,b\le 2$ as trivial when $\min(a,b)=0$ and of
dimension $a+b-1$ when $1\le a,b\le 2$.
E.g., for $a=b=1$ the only linear combination of $uvw$, $u^2w$, $vw^2$, $uv^2$
that belongs to the space inducated for $b=1$ is,
up to scalar multiplication, $u^2w+\zeta vw^2+\zeta^2 uv^2$.
We deduce that $\mathcal{S}\cap \mathcal{T}$ is the free $k[s,t]$-module with basis
\begin{gather*}
t^3uvw,\ \ st(u^2w+\zeta vw^2+\zeta^2 uv^2),\ \ st^2(u^2w+\zeta^2vw^2+\zeta uv^2),\\
st^3(u^2w+vw^2+uv^2),
\ \ s^2t(uw^2+\zeta v^2w+\zeta^2u^2v),\\
s^2t^2(uw^2+\zeta^2v^2w+\zeta u^2v),\ \ s^2t^3(uw^2+v^2w+u^2v),\\
s^3(u^3+v^3+w^3-3uvw),
\ \ s^3t(u^3+\zeta^2v^3+\zeta w^3),\ \ s^3t^2(u^3+\zeta v^3+\zeta^2w^3).
\end{gather*}

By Proposition \ref{prop.10sections}, the basis of
$\mathcal{S}\cap \mathcal{T}$ defines a morphism
$X'\to \A^2\times \PP^9$.
With linear algebra, we find $f_1$, $\dots$, $f_{27}\in k[s,t,x_0,\dots,x_9]$
that are
\begin{itemize}
\item homogeneous of degree $2$ in
$x_0$, $\dots$, $x_9$ and at most linear in $s$, $t$.
\item zero upon substituting $s^3$ for $s$, $t^3$ for $t$,
and the $i$th basis element for $x_{i-1}$, for every $i$,
\item linearly independent over $k(s,t)$.
\end{itemize}
The polynomials are displayed in Table \ref{definingequations}.
The conditions imply that
$$
\Proj(k(s,t)[x_0,\dots,x_9]/(f_1,\dots,f_{27}))
$$ 
is the
Brauer-Severi surface that appears as generic fiber of $\pi$.
\begin{table}
\[
\begin{array}{r|c}
i&f_i \\
\hline
1&-3\zeta^2x_0x_4-x_1x_3+x_2^2\\
2&-tx_1^2+3\zeta x_0x_5+x_2x_3\\
3&-tx_1x_2-3x_0x_6+x_3^2\\
4&(\zeta-\zeta^2)x_0x_7-x_1x_5+x_2x_4\\
5&(1-\zeta^2)x_0x_8-x_1x_6+x_3x_4\\
6&-(\zeta-\zeta^2)x_0x_9-tx_1x_4+x_3x_5\\
7&\zeta^2x_1x_6+\zeta x_3x_4+x_2x_5\\
8&\frac{1-\zeta^2}{3}x_1x_8-\frac{1-\zeta^2}{3}x_2x_7+x_4^2\\
9&\frac{1-\zeta}{3}x_1x_9-\frac{1-\zeta}{3}x_3x_7+x_4x_5\\
10&\zeta^2tx_1x_4+\zeta x_3x_5+x_2x_6\\
11&-x_3x_7+(1-\zeta)x_4x_5+x_2x_8\\
12&-tx_1x_7+(1-\zeta^2)x_5^2+x_2x_9\\
13&-9sx_0^2+\zeta tx_1x_5+\zeta^2tx_2x_4+x_3x_6\\
14&-3(1-\zeta^2)sx_0x_1-tx_1x_7+(1-\zeta)x_5^2+x_3x_8\\
15&-3(1-\zeta)sx_0x_2-tx_2x_7+(1-\zeta^2)tx_4^2+x_3x_9\\
16&3\zeta sx_0x_1-x_5^2+x_4x_6\\
17&-(1-\zeta)sx_1^2-\zeta x_5x_7+x_4x_8\\
18&-(1-\zeta^2)sx_1x_2-\zeta^2x_6x_7+x_4x_9\\
19&3\zeta^2 sx_0x_2-tx_4^2+x_5x_6\\
20&-(1-\zeta)sx_1x_2-\zeta x_6x_7+x_5x_8\\
21&-(1-\zeta^2)sx_2^2-\zeta^2 tx_4x_7+x_5x_9\\
22&-3sx_0x_3-tx_4x_5+x_6^2\\
23&-(1-\zeta^2)sx_1x_3+(\zeta-\zeta^2)sx_2^2-\zeta tx_4x_7+x_6x_8\\
24&-(1-\zeta)sx_2x_3-\zeta tx_4x_8+x_6x_9\\
25&-3\zeta sx_1x_4-x_7x_9+x_8^2\\
26&-3\zeta^2 sx_1x_5-3\zeta sx_2x_4-tx_7^2+x_8x_9\\
27&-3\zeta^2 sx_2x_5-tx_7x_8+x_9^2
\end{array}
\]
\caption{The polynomials $f_1$, $\dots$, $f_{27}\in k[s,t,x_0,\dots,x_9]$.}
\label{definingequations}
\end{table}

\begin{lemm}
\label{lem.flat9}
Let $f_1$, $\dots$, $f_{27}$ be as in Table \ref{definingequations},
where $k$ is a field of characteristic different from $3$ with primitive
cube root of unity $\zeta$.
Then $k[s,t,x_0,\dots,x_9]/(f_1,\dots,f_{27})$ is finitely generated and free
as a module over $k[s,t,x_0,x_1,x_7]$, with basis
$1$, $x_2$, $x_3$, $x_4$, $x_5$, $x_5^2$, $x_6$, $x_8$, $x_9$.
\end{lemm}

\begin{proof}
Using the relations $f_i=0$ for $i=1$, $\dots$, $27$ we
write down $9\times 9$-matrices for the action of
multiplication by $x_i$ on the claimed basis elements
for $i\in \{2,3,4,5,6,8,9\}$.
The matrices, as may be checked, commute pairwise and
obey relations corresponding to $f_1$, $\dots$, $f_{27}$.
So the evident homomorphism
from $k[s,t,x_0,\dots,x_9]/(f_1,\dots,f_{27})$
to the free $k[s,t,x_0,x_1,x_7]$-module of rank $9$ is an isomorphism.
\end{proof}

Lemma \ref{lem.flat9} implies that the ideal
$(f_1,\dots,f_{27})$ of $k[s,t,x_0,\dots,x_9]$ is prime.
Indeed, the ideal is saturated with respect to
$x_0$, $\dots$, $x_9$ (e.g., $x_0g\in (f_1,\dots,f_{27})$ implies
$g\in (f_1,\dots,f_{27})$ for $g\in k[s,t,x_0,\dots,x_9]$),
and the ideal generated by $f_1$, $\dots$, $f_{27}$ in
$k(s,t)[x_0,\dots,x_9]$ defines
an anticanonically embedded Brauer-Severi variety.
So the latter ideal, and hence also the former ideal, is prime.

\begin{prop}
\label{prop.flat9}
Let $k$ be a field of characteristic different from $3$ with primitive
cube root of unity $\zeta$,
$\pi\colon X\to \A^2\smallsetminus\{0\}$ and
$\pi'\colon X'\to \A^2\smallsetminus\{0\}$ as in
\eqref{eqn.brauerseveribundleX}--\eqref{eqn.brauerseveribundleXprime} and
$f_1$, $\dots$, $f_{27}\in k[s,t,x_0,\dots,x_9]$ as in
Table \ref{definingequations}.
The coherent sheaf $\pi_*\omega_\pi^\vee$ is free of rank $10$,
and for some basis
the scheme-theoretic image in $\A^2\times \PP^9$ is defined by
the prime ideal
$(f_1,\dots,f_{27})$ and is a nonsingular variety, flat over $\A^2$.
For an explicit basis of the free coherent sheaf $\pi'_*\omega_{\pi'}^\vee$
of rank $10$
the scheme-theoretic image in $\A^2\times \PP^9$ is defined by the
ideal obtained from $(f_1,\dots,f_{27})$ by applying
$s\mapsto s^3$ and $t\mapsto t^3$, which is also prime.
The fiber over $s=t=0$ is an irreducible, nonreduced scheme whose
reduced subscheme is a cone over a twisted cubic curve.
\end{prop}

The explicit basis mentioned in the statement of Proposition \ref{prop.flat9}
is the one given above, as the basis of $\mathcal{S}\cap \mathcal{T}$.

\begin{proof}
As remarked after Lemma \ref{lem.flat9}, the ideal generated by
$f_1$, $\dots$, $f_{27}$ is prime, and the same holds after
applying $s\mapsto s^3$ and $t\mapsto t^3$.
These facts, plus the remark
about $\Proj(k(s,t)[x_0,\dots,x_9]/(f_1,\dots,f_{27}))$ made before
Lemma \ref{lem.flat9},
justify the statements about scheme-theoretic image.

We directly analyze the fiber over $s=t=0$.
First, the reduced subscheme is contained in the locus
$x_3=x_5=x_6=x_8=x_9=0$, as we see directly from the defining equations.
The defining equations with 
$$
s=t=x_3=x_5=x_6=x_8=x_9=0
$$ 
no longer involve
$x_1$.
There are three equations, quadratic in $x_0$, $x_2$, $x_4$, $x_7$, the
equations of a twisted cubic.
So the reduced subscheme of the fiber is a cone over a twisted cubic.
The coordinate point 
$$
p_0=(1:0:0:0:0:0:0:0:0:0)
$$ 
lies on the fiber and has
Zariski tangent space of dimension $3$.
But $p_0$ is smooth point of the reduced subscheme.
So, the fiber is nonreduced.

Now we consider the variety
$$
\Proj(k[s,t,x_0,\dots,x_9]/(f_1,\dots,f_{27})).
$$
Flatness over $\A^2$ follows from Lemma \ref{lem.flat9}.
It remains to show that this is a nonsingular variety.
For this, we let the two-dimensional torus $\G_m^2$ act on
$\A^2$ by $(\delta,\varepsilon)\cdot (s,t)=(\delta^3s,\varepsilon^3t)$ and
on $\PP^9$ by sending $(x_0:\dots:x_9)$ to
\[
(\varepsilon^3x_0:\delta\varepsilon x_1:\delta\varepsilon^2x_2:\delta\varepsilon^3x_3:
\delta^2\varepsilon x_4:\delta^2\varepsilon^2x_5:\delta^2\varepsilon^3x_6:\delta^3x_7:
\delta^3\varepsilon x_8:\delta^3\varepsilon^2x_9)
\]
and thereby obtain a torus action on the
variety defined by $f_1$, $\dots$, $f_{27}$.
Suppose the variety is singular.
Since the singular locus is closed and torus-invariant, it must
contain a fixed point, i.e., a coordinate point, of the fiber over $s=t=0$.
From the description above, we only have to consider the
coordinate points $p_0$, $p_1$, and $p_7$.
For each we specify $7$ defining equations and $7$ variables:
\[
\begin{array}{c|c|c}
\text{point}&\text{equations}&\text{variables} \\ \hline
p_0&f_1,f_2,f_3,f_4,f_5,f_6,f_{13}&s,x_4,x_5,x_6,x_7,x_8,x_9\\
p_1&f_1,f_2,f_4,f_5,f_8,f_9,f_{17}&s,t,x_3,x_5,x_6,x_8,x_9\\
p_7&f_4,f_8,f_9,f_{17},f_{18},f_{25},f_{26}&t,x_0,x_2,x_3,x_5,x_6,x_9
\end{array}
\]
By direct computation we see that
the corresponding Jacobian matrix at each $p_i$ is nonsingular, and thus
$p_i$ is a nonsingular point of the variety defined by
$f_1$, $\dots$, $f_{27}$.
\end{proof}

The proof of Theorem~\ref{thm:P2bundle} proceeds as in the case of conic bundles \cite{oesinghaus}:
\begin{itemize}
\item Represent the generic fiber of $\pi$ by a
$3$-torsion Brauer group element $\alpha\in \Br(k(S))$. 
Apply embedded resolution of singularities
so that $\alpha$ is ramified along a
simple normal crossing divisor $D_1\cup\dots\cup D_\ell$. 
Let $\mathcal{R}:=\sqrt[3]{(S,\{D_1,\dots,D_\ell\})}$ be the iterated root stack; then
by Remark \ref{rem.ramification}, $\alpha$ is the restriction of an element
\[ \beta\in \Br(\mathcal{R}). \]
Pass to a $\mu_3$-gerbe $\mathcal{G}\to \mathcal{R}$
whose class is a lift $\beta_0\in H^2(\mathcal{R},\mu_3)$ of
$\beta$.
The restriction of $\pi$ to $\pi^{-1}(U)$ is a smooth
$\PP^2$-bundle by \cite[Lem.\ 1.7]{hironakasmoothing}.
So, there is a corresponding sheaf of
Azumaya algebras over $U$ and a locally free sheaf of rank $3$ over
$\mathcal{G}\times_{\mathcal{R}}U$.
This spreads out to a coherent sheaf $\mathcal{E}$ on $\mathcal{G}$, which
we may take to be reflexive, hence
\cite{hartshornereflexive} only failing to be locally free, if at all, on a locus
of codimension $\ge 3$.
\item By a standard construction (closure in Grassmannian of rank $3$
quotients of $\mathcal{E}$), followed by desingularization and
destackification \cite{bergh}, we may suppose  that $\mathcal{E}$ is locally free
and $\mathcal{R}$ is an iterated root stack
$\sqrt[3]{(\widetilde{S},\{\widetilde{D}_1,\dots,\widetilde{D}_m\})}$ over a
smooth projective variety $\widetilde{S}$.
We may suppose, for every $i$, that over a general point of $\widetilde{D}_i$
the stabilizer action on the fiber of $\PP(\mathcal{E})$ is balanced,
since otherwise it may be made trivial with
suitable elementary transformations.
\item After blowing up intersections of pairs and triples of divisors,
we may arrange that no triple of divisors intersects, and for every
pair of intersecting divisors, the corresponding
projective representation $\mu_3\times \mu_3\to PGL_3$ over a
general point of the intersection is faithful.
The group-theoretic input to these assertions is the fact
that no more than two independent commuting copies of $\mu_3$ may be
found in $PGL_3$.
\item Apply Proposition \ref{prop.DP9contraction} to the
complement of the intersections of pairs of divisors, and
Propositions \ref{prop.10sections} and \ref{prop.flat9} and
Lemma \ref{lem.liftisohenselian} to fill in the fibers over
intersections of divisors.
\end{itemize}

\section{Application to rationality}
\label{sec.rationality}

Here we prove Theorem \ref{dp2thm}.
We use a theorem of Voisin \cite[Thm.\ 2.1]{voisin}, in the form of \cite[Thm.\ 2.3]{CTP}:

\begin{theo}
\label{thm:voisin}
Let $k$ be an uncountable algebraically closed field, $B$ a variety over $k$, and
$\phi\colon \mathcal{X}\to B$ a
flat projective morphism with integral fibers and smooth generic fiber. 
Suppose that there exists a $b_0\in B$ such that the fiber $X_0:=\phi^{-1}(b_0)$ satisfies:
\begin{itemize}
\item[(O)] for some positive integer $n$, invertible in $k$,
there is nontrivial unramified cohomology of $\mu_n$ on $X_0$;
\item[(R)] $X_0$ admits a universally $\mathrm{CH}_0$-trivial resolution of singularities. 
\end{itemize}
Then, for very general $b\in B$, the fiber $X_b:=\phi^{-1}(b)$ 
is not stably rational. 
\end{theo}

The reader may turn to \cite{CTP} for the definitions of
\emph{universally $\mathrm{CH}_0$-trivial scheme}
and \emph{universally $\mathrm{CH}_0$-trivial morphism}.
For a proper geometrically integral scheme $X$ over a field $\ell$, say of
dimension $d$,
a related notion (equivalent, when $X$ is smooth \cite[Prop.\ 1.4]{CTP}),
is for $X$ to admit a \emph{Chow decomposition of the diagonal}.
This means that in the Chow group $\mathrm{CH}_d(X\times X)$ of
$d$-dimensional cycles modulo rational equivalence,
the class $[\Delta_X]$ of the diagonal of $X$ is equal to a sum
\[ Z+[X]\times z_0 \]
for some $d$-dimensional cycle $Z$ supported on $D\times X$, for a closed subset
$D\subsetneq X$, and $0$-cycle $z_0$ of degree $1$ on $X$.
If we let
$\delta_X\in \mathrm{CH}_0(X_{\ell(X)})$ denote the restriction of $\Delta_X$ to
$X_{\ell(X)}:=\Spec(\ell(X))\times_{\Spec(\ell)} X$, this is equivalent to:
\[
{(*)}\qquad \text{$X$ has a $0$-cycle $z_0$ of degree $1$ and
$[\delta_X]=z_0$ in $\mathrm{CH}_0(X_{\ell(X)})$}.
\]
Condition (R), with $\mathrm{CH}_0$-trivial resolution of singularities
$\widetilde{X}_0\to X_0$, implies the equivalence of $(*)$ on
$X_0$ and $\widetilde{X}_0$, hence the equivalence of
admitting a Chow decomposition of the diagonal.
Chow decomposition of the diagonal is
obstructed by Condition (O) but implied by stable rationality,
while an argument with Chow schemes \cite[App.\ B]{CTP}
(Chow functors in case $\chara(k)>0$ \cite[\S 3]{HKTconic})
shows that the $k$-points whose fibers admit
a Chow decomposition of the diagonal occupy a countable union of
closed subsets of $B$.
An alternative approach, bypassing Chow groups,
appears in recent work \cite{nicaiseshinder}, \cite{kontsevichtschinkel}.

We return to Theorem \ref{dp2thm}, with $\chara(k)\ne 2$, $3$, and
smooth degree $2$ del Pezzo surface $S$.

We first treat the case $d=3$ and then explain the modifications to
deal with general $d$.
The anticanonical morphism expresses $S$ as
double cover of $\PP^2$ branched along a nonsingular quartic curve $R$,
which we take to have
defining equation $h$ with respect to some coordinates.

\medskip

{\em Step 1.} There exists a nonsingular cubic curve $D_0$ meeting $R$ in six tangencies
lying on a conic.
Among the $28$ bitangent lines to $R$, at most $12$ are
hyperflexes, since the weighted count of Weierstrass points of $R$ is $24$
and each hyperflex contributes $2$; cf.\ \cite{SV}.
Bitangent lines are identified with odd theta characteristics,
and it is known classically (see, e.g., \cite[\S 2]{grossharris}) that
any set of $8$ bitangent lines, which we may suppose disjoint from the set of hyperflexes,
contains a syzygetic triple $\ell$, $\ell'$, $\ell''$:
the six points of intersection with $R$ lie on a conic $f=0$.
Using the same symbols for defining linear forms and suitably rescaling $h$,
we obtain by Max Noether's theorem a relation
$h=f^2+\ell\ell'\ell''\ell'''$ for a fourth bitangent $\ell'''$.
Let $j$ be a line in general position with respect to
$\ell$, $\ell'$, $\ell''$ with suitably scaled defining equation.
Then we easily see that the cubic $g:=\ell\ell'\ell''-2fj-j^2\ell'''$,
appearing in the relation $h=(f+j\ell''')^2+g\ell'''$,
is nonsingular and meets $R$ in six tangencies.

\medskip

{\em Step 2.} Let $\widetilde{D}_0\to D_0$ be a nontrivial cyclic degree $3$ \'etale cover,
determining (Remark \ref{rem.purity}) a $3$-torsion element of $\Br(\PP^2\smallsetminus D_0)$,
which extends (Remark~\ref{rem.ramification}) to an element of $\Br(\sqrt[3]{(\PP^2,D_0)})$,
represented (\cite{artin}, \cite{dejong}) by an index $3$ division algebra at the generic point. 
As in the proof of Theorem~\ref{thm:P2bundle}, it
spreads out to a sheaf of Azumaya algebras over $\sqrt[3]{(\PP^2,D_0)}$.
We obtain an element $\gamma_0\in H^2(\sqrt[3]{(\PP^2,D_0)},\mu_3)$,
gerbe $G_0\to \sqrt[3]{(\PP^2,D_0)}$ with this class, and
locally free coherent sheaf $E_0$ of rank $3$ on $G_0$ such that the
pullback of the sheaf of Azumaya algebras is $\End(E_0)$.
We perform base change via $S\to \PP^2$ to obtain \emph{reducible} curve $D=D_1\cup D_2$ with
\'etale degree $3$ cover $\widetilde{D}\to D$,
element $\gamma\in H^2(\sqrt[3]{(S,D)},\mu_3)$,
gerbe $G\to \sqrt[3]{(S,D)}$, and locally free coherent sheaf $E$ on $G$.

\medskip

{\em Step 3.} Apply the deformation-theoretic machinery of \cite[\S 4.3]{HKTconic} to exhibit:
\begin{itemize}
\item an exact sequence $0\to \widetilde{E}\to E\to Q\to 0$, where $Q$ is supported on a divisor
and is locally free there of rank $2$, such that the
space of obstructions for $\widetilde{E}$ vanishes.
\item an \'etale neighborhood $B$ of the point $b_0$, corresponding to divisor $D$ and
cyclic cover $\widetilde{D}$, in the space of reduced nodal curves on $S$ with cyclic degree $3$
\'etale cover, corresponding divisor $\mathcal{D}\subset B\times S$,
root stack $\sqrt[3]{(B\times S,\mathcal{D})}$,
class $\Gamma\in H^2(\sqrt[3]{(B\times S,\mathcal{D})},\mu_3)$ restricting to $\gamma$,
gerbe $\mathcal{G}\to \sqrt[3]{(B\times S,\mathcal{D})}$ restricting to $G$,
locally free coherent sheaf $\widetilde{\mathcal{E}}$ restricting to $\widetilde{E}$.
\item a smooth $\PP^2$-bundle $P\to \sqrt[3]{(B\times S,\mathcal{D})}$,
which upon base change to $\mathcal{G}$ yields $\PP(\widetilde{\mathcal{E}})$.
\end{itemize}

\medskip

{\em Step 4.} Apply Proposition \ref{prop.DP9contraction} to
$P$
to obtain a Brauer-Severi bundle $\pi\colon \mathcal{X}\to B\times S$
satisfying the hypotheses of Theorem \ref{thm:voisin}.
Since, by construction, the discriminant curve of the
Brauer-Severi bundle $X_0\to \{b_0\}\times S$
has two components and the
Brauer class is given by cyclic covers, nontrivial on each component
and \'etale over the nodes, condition (O) is satisfied.
Condition (R) holds, since the singularities of $X_0$
(which lie over the nodes of the discriminant curve) are of toric type.

\medskip

The case $d>3$ is dealt with by taking $D_3$ to be the pre-image in $S$
of a general curve $C\subset \PP^2$ of degree $d-3$ and defining
$D'_0:=D_0\cup C$.
We pull back $\gamma_0\in H^2(\sqrt[3]{(\PP^2,D_0)},\mu_3)$ and the gerbe $G_0$ to
obtain $\gamma'_0\in H^2(\sqrt[3]{(\PP^2,D'_0)}^{\mathrm{sm}},\mu_3)$ and
a gerbe $G'_0$.
The pull-back of the locally free coherent sheaf $E_0$ has
trivial $\mu_3$-stabilizer actions over points of $C$.
By performing suitable elementary transformations over the
component of the gerbe of the root stack over $C$,
we may obtain balanced actions, compatible with
those over $D_0$.
Then the smooth $\PP^2$-bundle corresponding to the locally free
sheaf has the same \'etale local isomorphism type as that
corresponding to $E$ (Step 2).
We thereby obtain an extension of $\gamma'_0$ to
$\gamma''_0\in H^2(\sqrt[3]{(\PP^2,D'_0)},\mu_3)$, of
$G'_0$ to $G''_0\to \sqrt[3]{(\PP^2,D'_0)}$, and of
$E$ to a locally free coherent sheaf on $G''_0$.
We proceed with Steps 3 and 4, taking $D$ to be the union of $D_1$, $D_2$, and
the pull-back $D_3$ of $C$;
the triviality of the cyclic cover over $D_3$ poses no problems in carrying
out these steps.

\appendix
\section{Birational contractions}
\label{app.birat}
Let $k$ be a field and $X$ a scheme, proper over $k$.
A line bundle $L$ on $X$ is called \emph{semiample}
if $L^m$ is globally generated for some $m>0$.
Then the set $M(X,L)$ of $m\ge 0$, such that $L^m$ is
globally generated, consists, for $m\gg 0$, of precisely the multiples
of a well-defined positive integer, the
\emph{exponent} of $M(X,L)$; cf.\ \cite[\S 2.1.B]{lazarsfeldI}.
We start by recalling the contraction determined by a semiample line bundle.

\begin{prop}
\label{prop.semiampleoverfield}
Let $X$ be a scheme, proper over a field $k$, with semiample line bundle $L$
and sections $s_0$, $\dots$, $s_M\in H^0(X,L^m)$ generating $L^m$ for some $m>0$,
thereby inducing a morphism $\psi\colon X\to \PP^M$.
Let $Y=\psi(X)$, and write the Stein factorization of $\psi$ as morphism
\[ \pi\colon X\to X' \]
followed by finite morphism $f\colon X'=\Spec(\psi_*\cO_X)\to Y$.
We let $\cO_Y(1)$, respectively $\cO_{X'}(1)$ denote the restriction to $Y$,
respectively the pull-back to $X'$ of the line bundle
$\cO_{\PP^M}(1)$, so $\pi^*\cO_{X'}(1)\cong L^m$.
Then:
\begin{itemize}
\item[(i)] The morphism $\pi\colon X\to X'$ is proper and
induces an isomorphism $\cO_{X'}\eqto\pi_*\cO_X$.
\item[(ii)] For all $n\gg 0$ the line bundle $\cO_{X'}(n)$ is very ample,
and for any positive integer $n$ such that $\cO_{X'}(n)$ is very ample,
the morphism $X\to \PP^N$ determined by a basis of $H^0(X,L^{mn})$
factors uniquely through $X'$, identifying $X'$ with the
image of $X$ in $\PP^N$.
\item[(iii)] With the property in \emph{(ii)} the scheme $X'$, with
$\pi\colon X\to X'$, is determined uniquely up to canonical isomorphism,
independently of the choice of $m$ and the generating sections, and the powers of $L$ that
are isomorphic to the pull-back of a line bundle from $X'$ are precisely all 
multiples of the exponent of $M(X,L)$.
\end{itemize}
\end{prop}

\begin{proof}
We have stated, in (i), a basic property of
the Stein factorization; see \cite[4.3.1]{EGAIII}.
The first statement in (ii) follows from standard facts on very ample line bundles
\cite[4.4.10(ii), 5.1.6]{EGAII}, and for the second statement,
using the isomorphism $\cO_{X'}\cong \pi_*\cO_X$ we have
\begin{equation}
\label{eqn.globalsectionsiso}
H^0(X,L^{mn})\cong H^0(X',\cO_{X'}(n)),
\end{equation}
so for $n$ such that $\cO_{X'}(n)$ is very ample, a choice of basis
compatibly maps $X$ to $\PP^N$ and embeds $X'$ in $\PP^N$.
The factorization is unique, and by (i), identifies $X'$ with the
image of $X\to \PP^N$.
For (iii), if $\tilde s_0$, $\dots$, $\tilde s_{\widetilde{M}}$ is another collection
of generating sections of $L^m$, leading to $\tilde\pi\colon X\to \widetilde{X}'$,
then for $n$ such that $\cO_{X'}(n)$ and $\cO_{\widetilde{X}'}(n)$ are very ample,
$X'$ and $\widetilde{X}'$ are both identified with the image of $X\to \PP^N$, as in (ii).
The isomorphism $X'\cong \widetilde{X}'$ obtained this way is independent of $n$:
it suffices to compare the isomorphisms arising from $n$, leading to $X\to \PP^N$,
and $dn$, for some $d>0$, leading to $X\to \PP^{N'}$, but these are compatible with
the $d$-tuple Veronese embedding of $\PP^N$ and the rational map from $\PP^{N'}$
corresponding to 
$$
\Sym^d(H^0(X,L^{mn}))\to H^0(X,L^{dmn}).
$$
When $\pi\colon X\to X'$ and $\tilde\pi\colon X\to \widetilde{X}'$ arise
from different powers of $L$ the argument is similar, with pairs
$n$ and $\tilde n$ of positive integers chosen proportionally to yield the same
power of $L$ in the respective isomorphisms \eqref{eqn.globalsectionsiso}
with $H^0(X',\cO_{X'}(n))$ and $H^0(\widetilde{X}',\cO_{\widetilde{X}'}(\tilde n))$.
Taking, as powers of $L$, two sufficiently large
consecutive multiples of the exponent of $M(X,L)$ yields respective line bundles
$\cO_{X'}(1)$, such that the pull-back of their difference is isomorphic to
$L^d$, where $d$ denotes the exponent of $M(X,L)$.
If a line bundle $L'$ on $X'$ satisfies $\pi^*L'\cong L^e$ for some $e>0$ then
$L'$ is ample, so all sufficiently large powers of $L'$ are globally generated and thus
$e$ is divisible by $d$.
\end{proof}

A morphism satisfying the properties stated in (i) is called a \emph{contraction}
and is necessarily surjective
with geometrically connected fibers \cite[4.3.2, 4.3.4]{EGAIII}.
We call this the \emph{contraction determined by $L$}.

Below we present a result which encompasses the following examples.

\begin{exam}
\label{exa.semiampleoverfield}
We suppose $X$ is smooth and projective over $k$ of
dimension $n$, with ample line bundle $L_0$, and effective divisor
$E\subset X$.
\begin{itemize}
\item[(i)]
If $E\cong \PP^{n-1}$ with normal bundle $N_{E/X}\cong \cO_{\PP^{n-1}}(-1)$,
and $m$ denotes the positive integer with $L_0|_E\cong \cO_{\PP^{n-1}}(m)$,
then $L:=L_0(mE)$ is semiample and determines a contraction
\[ X\cong B\ell_{\{p\}}X'\to X' \]
of $E$ to a $k$-point $p$ on a nonsingular projective variety $X'$.
\item[(ii)]
If $n=3$ and $E\cong \PP^1\times \PP^1$, with
$N_{E/X}\cong \cO_{\PP^1\times \PP^1}(-1,-1)$
and $L_0|_E\cong \cO_{\PP^1\times \PP^1}(m_1,m_2)$, then
$L:=L_0(\min(m_1,m_2)E)$ is semiample and determines a contraction
\[ X\cong B\ell_{C}X'\to X' \]
for a curve $C\cong \PP^1$ on nonsingular $X'$ with
$N_{C/X'}\cong \cO_{\PP^1}(-1)^2$, when $m_1\ne m_2$, respectively
\[ X\cong B\ell_{\{p\}}X'\to X' \]
for a $k$-point $p\in X'$ with an ordinary double point singularity,
when $m_1=m_2$.
\end{itemize}
\end{exam}

When $\dim(X)=2$, Example \ref{exa.semiampleoverfield}(i) is
the contraction theorem of Castelnuovo and Enriques and leads
to minimal models of algebraic surfaces.
The case of dimension at least $3$ requires the sophistication of the
minimal model program (see \cite{KMM} for a survey), including the
observation that a curve class (modulo numerical equivalence), which meets
a canonical divisor positively and is
extremal in the cone of effective classes, uniquely determines a contraction.
Mori, in dimension $3$, identifies
five kinds of divisorial contractions \cite[Thm.\ 3.3]{morithreefolds},
including the ones given in
Example \ref{exa.semiampleoverfield}; we see already the possibility
of contraction to a singular point, which is essential for the theory.
The results below
on divisorial contractions to a point
(Proposition \ref{prop.semiampleoverfieldtopoint}) and to a general subscheme
(Proposition \ref{prop.semiampleoverfieldtoscheme}) are essentially
contained in \cite[Lem.\ 3.32]{morithreefolds}; a relative version is
presented in Proposition \ref{prop.semiampleoverStoscheme}.

An important ingredient,
already appearing in loc.\ cit., is the following notion \cite{mumfordquadratic}:
an ample line bundle $L$ on a projective scheme $X$
is \emph{normally generated} if
$$
\Sym^nH^0(X,L)\to H^0(X,L^n)
$$ 
is surjective for every positive integer $n$.

\begin{prop}
\label{prop.semiampleoverfieldtopoint}
Let $X$ be a projective scheme
over $k$ with ample line bundle $L_0$
and effective Cartier divisor $E\subset X$ with $H^0(E,\cO_E)=k$ (this is
the case, e.g., when geometrically $E$ is connected and reduced).
Suppose that $N_{E/X}^\vee$ is
ample, with $L_0|_E\otimes (N_{E/X})^m$ a torsion line bundle for some $m>0$.
Assume, furthermore, that $N_{E/X}^\vee$ is normally generated, and
$H^1(E,(N_{E/X}^\vee)^n)=0$ for all $n>0$.
Then $L:=L_0(mE)$ is semiample with exponent of $M(X,L)$ equal to the order
of $L_0|_E\otimes (N_{E/X})^m$ in $\Pic(E)$ and determines a contraction
$X\cong B\ell_{\{p\}}Y\to Y$
of $E$ to a $k$-point $p\in Y$ such that the local ring $\cO_{Y,p}$ has
associated graded $k$-algebra isomorphic to $\bigoplus H^0(E,(N_{E/X}^\vee)^n)$.
\end{prop}

\begin{proof}
Replacing $L_0$ by some power, we may suppose $L_0$ is very ample with
$H^1(X,L_0)=0$, the restriction of $L$ to $E$ is trivial, and $m\ge 2$.
Now, we claim, $H^1(X,L_0(jE))=0$ for $0\le j<m$.
We have this for $j=0$, and we argue by induction on $j$ using the
exact sequence in cohomology associated with
\[
0 \to L_0((j-1)E)\to L_0(jE)\to (N_{E/X}^\vee)^{m-j}\to 0.
\]
Then, by considering the sequence for $j=m$ we find that $L$ has a
global section which restricts to a nonzero constant on $E$
(under $L|_E\cong \cO_E$).
Furthermore, by the sequence for $j=m-1$ we may lift a basis of
$H^0(E,N_{E/X}^\vee)$ to global sections of $L$,
vanishing along $E$.
We may view, as well, all the global sections of $L_0$ as global sections of $L$.
So we see that $L$ is globally generated, and
the induced morphism to its image $Y$ in
projective space collapses $E$ to a point $p$,
such that the scheme-theoretic pre-image of $p$ is $E$, and
$X\smallsetminus E\cong Y\smallsetminus \{p\}$.
Let us denote by $\cI_E$ the ideal sheaf of $E$, so
$\cI_E/\cI_E^2\cong N_{E/X}^\vee$.
Then we see from these two cases of the sequence,
$$
\mathfrak{m}_p/\mathfrak{m}_p^2\to H^0(E,\cI_E/\cI_E^2)
$$
is surjective.
Again replacing $L_0$ with a tensor power
we may suppose that the morphism
$\pi\colon X\to Y$ given by the global sections of $L$
induces an isomorphism $\cO_Y\eqto \pi_*\cO_X$.
Then by the Theorem on Formal Functions, we have an isomorphism
\begin{equation}
\label{eqn.semiampleoverfieldtopointTFF}
\widehat{\cO}_{Y,p}\eqto \varprojlim H^0(X,\cO_X/\cI_E^n).
\end{equation}
We claim, the isomorphism \eqref{eqn.semiampleoverfieldtopointTFF}
comes from isomorphisms
\[ \cO_{Y,p}/\mathfrak{m}_p^n\to H^0(X,\cO_X/\cI_E^n) \]
for every $n$.

We begin by observing, using the $H^1$-vanishing hypothesis, that
\begin{equation}
\label{eqn.semiampleoverfieldtopointsurj}
H^0(X,\cI_E^i/\cI_E^{j+1})\to H^0(X,\cI_E^i/\cI_E^j)
\end{equation}
is surjective for all $i<j$.
We will use this fact several times in the remainder of the proof.

By the normal generation hypothesis, for every $j$, the homomorphism
$$
H^0(E,\cI_E/\cI_E^2)\otimes H^0(E,\cI_E^{j-1}/\cI_E^j)\to H^0(E,\cI_E^j/\cI_E^{j+1}),
$$
given by multiplication, is surjective.
Combining this with the observation about $\mathfrak{m}_p/\mathfrak{m}_p^2$, we have
the surjectivity of
$$
\mathfrak{m}_p/\mathfrak{m}_p^2\otimes H^0(E,\cI_E^{j-1}/\cI_E^j)\to H^0(E,\cI_E^j/\cI_E^{j+1}).
$$
We see more generally that if $s_1$, $\dots$, $s_\ell$ denote lifts of a
basis of $\mathfrak{m}_p/\mathfrak{m}_p^2$ to the coordinate ring of an
affine neighborhood of $p$, then multiplication by $s_1$, $\dots$, $s_\ell$
induces surjective
\begin{equation}
\label{eqn.semiampleoverfieldtopointmoregenrl}
H^0(X,\cI_E^i/\cI_E^j)^\ell \to H^0(X,\cI_E^{i+1}/\cI_E^{j+1})
\end{equation}
for all $i<j$.
Indeed, we have the
exact sequences
\begin{equation}
\label{eqn.semiampleoverfieldtopointexactseq1}
0\to H^0(X,\cI_E^{j-1}/\cI_E^j)^\ell\to H^0(X,\cI_E^i/\cI_E^j)^\ell
\to H^0(X,\cI_E^i/\cI_E^{j-1})^\ell\to 0
\end{equation}
and
\begin{equation}
\label{eqn.semiampleoverfieldtopointexactseq2}
0\to H^0(X,\cI_E^j/\cI_E^{j+1})\to H^0(X,\cI_E^{i+1}/\cI_E^{j+1})
\to H^0(X,\cI_E^{i+1}/\cI_E^j)\to 0,
\end{equation}
with compatible vertical maps given by
multiplication with $s_1$, $\dots$, $s_\ell$.
Induction on $j$ for each fixed $i$ establishes the surjectivity of
\eqref{eqn.semiampleoverfieldtopointmoregenrl}.

If we regard $i$ as fixed and
apply to \eqref{eqn.semiampleoverfieldtopointmoregenrl} the general fact
that an inverse system of surjective linear maps of finite-dimensional
vector spaces induces a surjective map of inverse limits, then we find,
\begin{equation}
\label{eqn.semiampleoverfieldtopointfurthermore}
\mathfrak{m}_p\varprojlim H^0(X,\cI_E^i/\cI_E^n)=\varprojlim H^0(X,\cI_E^{i+1}/\cI_E^n).
\end{equation}
An inductive argument now shows that we have an isomorphism
\begin{equation}
\label{eqn.semiampleoverfieldtopointTFFp}
\mathfrak{m}_p^i\widehat{\cO}_{Y,p}\eqto \varprojlim H^0(X,\cI_E^i/\cI_E^n),
\end{equation}
for every $i$.
The case $i=0$ is \eqref{eqn.semiampleoverfieldtopointTFF}, and from
\eqref{eqn.semiampleoverfieldtopointTFFp} we have an isomorphism of the
respective ideals generated on each side by $\mathfrak{m}_p$,
which with \eqref{eqn.semiampleoverfieldtopointfurthermore} gives us the
inductive step.

The claim is now immediate from \eqref{eqn.semiampleoverfieldtopointTFF}
and \eqref{eqn.semiampleoverfieldtopointTFFp}.
Notice that the claim implies that for every $n$ we have an isomorphism
\begin{equation}
\label{eqn.semiampleoverfieldtopointTFFb}
\mathfrak{m}_p^{n-1}/\mathfrak{m}_p^n\to H^0(E,\cI_E^{n-1}/\cI_E^n).
\end{equation}
The statement concerning the associated graded ring of $\cO_{Y,p}$
in the proposition is immediate from
\eqref{eqn.semiampleoverfieldtopointTFFb}.
By the universal property of blowing up,
we obtain a unique morphism $X\to B\ell_{\{p\}}Y$ factorizing $\pi$,
whose restriction to $E$ is given by
\[
E\cong \Proj\big(\bigoplus H^0(E,\cI_E^n/\cI_E^{n+1})\big)\eqto
\Proj\big(\bigoplus \mathfrak{m}_p^n/\mathfrak{m}_p^{n+1}\big).
\]
Now $X\to B\ell_{\{p\}}Y$ is a morphism of projective schemes,
such that the scheme-theoretic pre-image of every point of $B\ell_{\{p\}}Y$
is a reduced point with the same residue field.
By \cite[18.12.6]{EGAIV}, this must be a closed immersion.
But the exceptional divisor of the blow-up is an effective Cartier divisor, over
the complement of which we have an isomorphism, hence
$X\to B\ell_{\{p\}}Y$ is an isomorphism.
\end{proof}

For the version of Proposition \ref{prop.semiampleoverfieldtopoint} which
allows contraction to a more general variety than a point, we need relative
versions of the hypotheses.
A relative version of the hypothesis $H^0(E,\cO_E)=k$ is to be given
a flat contraction morphism $E\to F$,
and the conclusion will
be that this is determined by $L|_E$.

\begin{lemm}
\label{lem.ngfamily}
Let $Y=\Spec(A)$ be a Noetherian affine scheme,
$\pi\colon X\to Y$ a flat projective morphism,
$L$ an ample line bundle on $X$, and $y\in Y$ a point such that
the restriction $L_y$ of $L$ to the fiber
$X_y$ is very ample (resp.\ is normally generated) and satisfies
$H^1(X_y,L_y^n)=0$, for $n=1$ (resp.\ for all $n>0$).
Then there exists an open subscheme $U\subset Y$ with $y\in U$ such that
relative to $X\times_YU\to U$ the restriction of $L$ is very ample
(resp.\ is very ample with 
$$
\Sym^n H^0(X\times_YU,L)\to H^0(X\times_YU,L^n)
$$
surjective for all $n>0$).
\end{lemm}

\begin{proof}
By the machinery of cohomology and base change, we may replace $Y$ by
a Zariski neighborhood of $y$, such that the coherent sheaf $\pi_*L$ is locally free
and its formation commutes with arbitrary base change.
Since $L_y$ is globally generated,
the cokernel of $\pi^*\pi_*L\to L$ has trivial fiber over $y$, hence may be
assumed to vanish if we shrink $Y$ further.
We obtain a morphism $X\to \PP((\pi_*L)^\vee)$ over $Y$,
and this restricts to a closed immersion of $X_y$ since $L_y$ is very ample.
The support of the cokernel of $\cO_{\PP^N_Y}\to \cO_X$ is disjoint from
$X_y$, hence by further shrinking we find that $L$ is relatively very ample.
There exists $n_0$ such that 
$$
\Sym^n H^0(X,L)\to H^0(X,L^n)
$$ is surjective
for all $n>n_0$ (consequence of the vanishing of $H^1$ of all sufficiently large twists
of the ideal sheaf of $X$ in $\PP(H^0(X,L)^\vee)$).
Now suppose that $L_y$ is normally generated.
Then, as above, we may suppose that $\pi_*L^n$ is locally free and its
formation commutes with arbitrary base change for all $1\le n\le n_0$.
For all such $n$, we consider the cokernel of $\Sym^n H^0(X,L)\to H^0(X,L^n)$,
which may be assumed to vanish by shrinking $Y$, and thus we have the desired
conclusion when $L_y$ is normally generated.
\end{proof}

\begin{lemm}
\label{lem.Hivanishpairample}
Let $X$ be a scheme, with projective morphism $\pi\colon X\to S$ to a Noetherian
scheme $S$ and relatively ample line bundles $L$ and $M$.
Then for every coherent sheaf $\cF$ on $X$ there exists
$n\in \N$, such that $R^i\pi_*(\cF\otimes L^a\otimes M^b)=0$ for all $i>0$ and
$a$, $b\in \N$ with $\max(a,b)\ge n$.
\end{lemm}

\begin{proof}
We immediately reduce to the case $S$ when is affine.
Using the ampleness of $L$ and $M$,
we reduce further to establishing the assertion with
$\max(a,b)$ replaced by $\min(a,b)$.
By considering $\cF$ twisted by various powers of $L$ and $M$
and replacing $L$ and $M$ with suitable powers,
we may suppose that $L$ and $M$ are very ample.
By pushing forward $\cF$ along the closed immersion
$X\to \PP_S^\ell\times_S \PP_S^m$ for some $\ell$, $m\in \N$ determined by
the sections of $L$ and $M$, we are reduced to proving the following assertion:
if $\cF$ is a
coherent sheaf on $\PP_S^\ell\times_S \PP_S^m$ then there exists $n\in \N$ such
that 
$$
H^i(\PP_S^\ell\times_S \PP_S^m,\cF(a,b))=0\quad \text{  for all $a$, $b\in \N$ with } \, \min(a,b)\ge n.
$$
Here, $\cF(a,b)$ denotes the twist by the pullback of $\cO_{\PP_S^\ell}(a)$ and
the pullback of $\cO_{\PP_S^m}(b)$.
The proof is by descending induction on $i$, starting from the case
$i=\ell+m+1$, which is trivial by the vanishing of cohomology in degrees
greater than the relative dimension.
For the inductive step, the ampleness of
$\cO_{\PP_S^\ell\times_S \PP_S^m}(1,1)$ implies the existence of $c$, $d\in \N$
such that there is an exact sequence
\[ 0\to \cG\to \cO_{\PP_S^\ell\times_S \PP_S^m}(-c,-c)^{\oplus d}\to \cF\to 0 \]
of coherent sheaves.
Twisting and applying the induction hypothesis and known cohomology of line bundles on
$\PP_S^\ell\times_S \PP_S^m$, we obtain the result.
\end{proof}

\begin{lemm}
\label{lem.Hivanishampleggen}
Let $X$ be a scheme, with projective morphism $\pi\colon X\to S$ to a Noetherian
scheme $S$ and line bundles $L$ and $M$,
such that locally on $S$ the line bundle $L$ is semiample, determining a
flat contraction $\psi\colon X\to Y$,
and $M$ is relatively ample.
Suppose $c$, $d\in \N$ are such that for all $y\in Y$ the fiber $X_y$ satisfies
$H^i(X_y,(M|_{X_y})^j)=0$ for all $1\le i\le c$ and $j\ge d$.
Then there exists $n\in \N$ such that 
$$
R^i\pi_*(L^a\otimes M^b)=0
$$ 
for all $1\le i\le c$
and $a$, $b\in\N$ such that $b\ge d$ and $\max(a,b)\ge n$.
\end{lemm}

\begin{proof}
As in the proof of the previous lemma, we may assume that $S$ is affine and
observe that the pair of maps to projective spaces,
determined by suitable powers of $L$ and $M$, leads to the vanishing of
all higher cohomology for $\min(a,b)\ge n$, for some $n\in \N$.
By ampleness, after possibly increasing $n$, we obtain the vanishing for
all $a$, $b\in \N$ with $b\ge n$.
Now we consider a fixed $b$ with $d\le b<n$ and $e$ with $0\le e<\ell$,
where $\ell$ is the exponent of $M(X,L)$, and apply the hypothesis, the
Leray spectral sequence, and the machinery of cohomology and base change
to obtain 
$$
H^i(X,L^{a\ell+e}\otimes M^b)\cong H^i(Y,\cO_Y(a)\otimes \psi_*(L^e\otimes M^b)),
$$
which by ampleness vanishes for sufficiently large $a$.
This gives the result.
\end{proof}

We recall \cite[6.10.1]{EGAIV}, a locally Noetherian scheme $Y$ is said to be
\emph{normally flat} along a closed subscheme $F$ if
$\cI_F^j/\cI_F^{j+1}$ is a locally free sheaf of $\cO_F$-modules for every $j$,
where $\cI_F$ denotes the ideal sheaf of $F$.

\begin{prop}
\label{prop.semiampleoverfieldtoscheme}
Let $X$ be a projective scheme over $k$ with
ample line bundle $L_0$ and effective Cartier divisor $E\subset X$.
Suppose for some $m>0$, the line bundle $L_0|_E\otimes (N_{E/X})^m$
is semiample and determines a flat contraction
$\psi\colon E\to F$.
Suppose, furthermore, $N_{E/X}^\vee$ is ample relative to $\psi$ and
on fibers is normally generated with vanishing $H^1$
of all positive powers.
Then $L:=L_0(mE)$ is semiample with exponent of $M(X,L)$ equal to
exponent of $M(E,L_0|_E\otimes(N_{E/X})^m)$ and determines a contraction
$X\cong B\ell_FY\to Y$ of $E$ to $F$ such that $Y$ is normally flat along $F$
with normal cone
\[ \Spec\Big(\bigoplus_{n\ge 0} \psi_*(N_{E/X}^\vee)^n\Big). \]
\end{prop}

\begin{proof}
Replacing $L_0$ by some power, we may suppose that $L_0$ is very ample,
with $H^1(X,L_0)=0$, $m\ge 2$, and $L|_E$ globally generated,
with sections defining $\psi\colon E\to F$ as in the statement.
Since the tensor product of a very ample line bundle and a globally generated
line bundle is very ample \cite[4.4.8]{EGAII},
for every $1\le a\le m-1$ the line bundle
$$
(L_0|_E)^{m-a}\otimes(L_0|_E\otimes N_{E/X}^m)^a\cong (L_0|_E)^m\otimes (N_{E/X})^{am}
$$
is very ample.
So, $L_0|_E\otimes (N_{E/X})^a$ is ample for $1\le a\le m-1$.
Now we claim, 
$$
H^1(E,(L_0|_E)^{n_1}\otimes (N_{E/X})^{n_2})=0
$$ 
for all
sufficiently large $n_1$ and $0\le n_2<mn_1$.
Indeed, Lemma \ref{lem.Hivanishpairample} applied to the pairs of line bundles
$L_0|_E\otimes (N_{E/X})^{a-1}$ and $L_0|_E\otimes (N_{E/X})^a$ for
$a=1$, $\dots$, $m-1$ takes care of the case $0\le n_2\le (m-1)n_1$.
We consider the pair of line bundles
$L|_E$ (semiample) and $M:=L_0|_E\otimes (N_{E/X})^{m-1}$ (ample).
Since the restriction of $L$ to any fiber of $\psi$ is trivial,
the restriction of $M$ is isomorphic to the restriction of $N_{E/X}^\vee$,
and by assumption $H^1$ of all positive powers vanish.
An application of Lemma \ref{lem.Hivanishampleggen} completes the verification
of the claim.
As a consequence, after again replacing $L_0$ by a power, we may suppose that
\begin{equation}
\label{eqn.semiampleoverfieldtoschemeH1}
H^1(E,L_0|_E\otimes(N_{E/X})^j)=0
\end{equation}
for $0\le j<m$.

The exact sequence
\[ 0\to L_0((j-1)E)\to L_0(jE)\to L_0|_E\otimes (N_{E/X})^j\to 0 \]
and an inductive argument lead to
$H^1(X,L_0(jE))=0$ for $0\le j<m$.
By the vanishing for $j=m-1$ and $j=m-2$, the homomorphisms
\[ H^0(X,L)\to H^0(E,L|_E)\quad\text{and}\quad
H^0(X,L(-E))\to H^0(E,L|_E\otimes N_{E/X}^\vee) \]
are surjective.
It follows that $L$ is globally generated, and if we let $\pi\colon X\to Y$
denote the contraction that it determines,
then the image of the restriction of $\pi$ to $E$ is a copy of $F$ in $Y$,
such that the scheme-theoretic pre-image of $F$ under $\pi$ is $E$.
Moreover, $\pi$ induces a surjective homomorphism of sheaves
\begin{equation}
\label{eqn.semiampleoverfieldtoschememoreover}
\cI_F/\cI_F^2\to \psi_*(\cI_E/\cI_E^2).
\end{equation}

Let $\mathfrak{Y}$ denotes the formal completion of $Y$ along $F$.
By the Theorem on Formal Functions in the form \cite[4.1.5]{EGAIII}, $\pi$ induces
an isomorphism
\begin{equation}
\label{eqn.semiampleoverfieldtoschemeTFF}
\cO_{\mathfrak{Y}}\eqto \varprojlim \pi_*(\cO_X/\cI_E^n).
\end{equation}
We claim that the isomorphism \eqref{eqn.semiampleoverfieldtoschemeTFF} comes
from isomorphisms
\begin{equation}
\label{eqn.semiampleoverfieldtoschemeTFFa}
\cO_Y/\cI_F^n\eqto \pi_*(\cO_X/\cI_E^n),
\end{equation}
for every $n$.

We start by observing, by the machinery of cohomology and base change,
that the hypothesis of vanishing of $H^1$ of positive powers of $N_{E/X}^\vee$
implies $R^1\psi_*(\cI_E^j/\cI_E^{j+1})=0$ for all $j>0$.
A consequence is the surjectivity of
\begin{equation}
\label{eqn.semiampleoverfieldtoschemesurj}
\pi_*(\cI_E^i/\cI_E^{j+1})\to \pi_*(\cI_E^i/\cI_E^j)
\end{equation}
for all $i<j$.
As a consequence, in the inverse system in
\eqref{eqn.semiampleoverfieldtoschemeTFF} the transition
maps are epimorphisms,
and the same holds with $\cO_X$ replaced by any positive power of $\cI_E$.

By Lemma \ref{lem.ngfamily}, the homomorphism
\[ \psi_*(\cI_E/\cI_E^2)\otimes \psi_*(\cI_E^{j-1}/\cI_E^j)\to
\psi_*(\cI_E^j/\cI_E^{j+1}) \]
induced by $\pi$ is surjective for all $j$.
Combining this with \eqref{eqn.semiampleoverfieldtoschememoreover} we have
the surjectivity of
$$
\cI_F/\cI_F^2\otimes \psi_*(\cI_E^{j-1}/\cI_E^j)\to \psi_*(\cI_E^j/\cI_E^{j+1}).
$$

The remainder of the argument follows closely the proof of
Proposition \ref{prop.semiampleoverfieldtopoint}.
If $V\subset Y$ is any affine open and
$s_1$, $\dots$, $s_\ell$ are sections that generate $\cI_F|_V$ as a sheaf of
$\cO_V$-modules, then multiplication by $s_1$, $\dots$, $s_\ell$ induces surjective
\begin{equation}
\label{eqn.semiampleoverfieldtoschememoregenrl}
H^0(U,\cI_E^i.\cI_E^j)^\ell\to H^0(U,\cI_E^{i+1}/\cI_E^{j+1})
\end{equation}
for all $i<j$,
where $U\subset X$ denotes the pre-image of $V$.
We see this by considering exact sequences,
analogous to \eqref{eqn.semiampleoverfieldtopointexactseq1}--\eqref{eqn.semiampleoverfieldtopointexactseq2},
but with sections on $U$ instead of $X$.

Fixing $i$ and letting $K_j$ denote the kernel of
\eqref{eqn.semiampleoverfieldtoschememoregenrl}, we write the corresponding
short exact sequences for consecutive values $j=n-1$ and $j=n$ and
connect them by homomorphisms of the form
\eqref{eqn.semiampleoverfieldtoschemesurj}, to see that
the induced homomorphism of kernels $K_n\to K_{n-1}$ is surjective.
After applying $\varprojlim$ we obtain a short exact sequence,
which tells us that
\[
\cI_F\varprojlim \pi_*(\cI_E^i/\cI_E^n)=
\varprojlim \pi_*(\cI_E^{i+1}/\cI_E^n).
\]
An inductive argument,
as in the proof of Proposition \ref{prop.semiampleoverfieldtopoint},
shows that we have an isomorphism
\begin{equation}
\label{eqn.semiampleoverfieldtoschemeTFFp}
\cI_F^i\cO_{\mathfrak{Y}}\eqto \varprojlim \pi_*(\cI_E^i/\cI_E^n),
\end{equation}
for every $i$.

The claim is now immediate from \eqref{eqn.semiampleoverfieldtoschemeTFF}
and \eqref{eqn.semiampleoverfieldtoschemeTFFp}.
The claim implies
that for every $n$ we have an isomorphism
\begin{equation}
\label{eqn.semiampleoverfieldtoschemeTFFb}
\cI_F^{n-1}/\cI_F^n\eqto \psi_*(\cI_E^{n-1}/\cI_E^n).
\end{equation}
The sheaf on the right is locally free by cohomology and base change,
hence $Y$ is normally flat along $F$.
The statement concerning the normal cone in the
proposition is immediate from \eqref{eqn.semiampleoverfieldtoschemeTFFb}.
The conclusion of the proof is exactly as in
Proposition \ref{prop.semiampleoverfieldtopoint}.
\end{proof}

The general result and examples are valid as well in a relative setting.
If $S$ is a Noetherian scheme and $X$ is a scheme, proper over $S$, then
we call a line bundle $L$ on $X$ \emph{relatively semiample} if over every
affine open subscheme of $S$ some positive power of $L$ is globally generated.
Equivalently, there exists $m>0$ such that over every affine open subscheme
$L^m$ is globally generated.
The set of natural numbers $m$ with this property will be denoted by $M(X/S,L)$,
or just $M(X,L)$ when no confusion will arise.
There is, as before, a well-defined exponent of $M(X/S,L)$.

\begin{prop}
\label{prop.semiampleoverS}
Let $X$ be a scheme with a proper morphism
$\varphi\colon X\to S$ to a Noetherian scheme $S$,
relatively semiample line bundle $L$, and surjective morphism of sheaves
$\varphi^*\cE\to L^m$, for some $m>0$, and coherent sheaf $\cE$ on $S$, thereby
inducing a morphism 
\[
\psi\colon X\to \Proj(\Sym^\bullet \cE).
\]
Let $Y=\psi(X)$, and write the Stein factorization of $\psi$ as
proper morphism
\[ \pi\colon X\to X', \]
followed by finite morphism $f\colon X'=\Spec(\psi_*\cO_X)\to Y$.
We let $\cO_Y(1)$, respectively $\cO_{X'}(1)$ denote the restriction to $Y$,
respectively the pull-back to $X'$ of $\cO_{\Proj(\Sym^\bullet \cE)}(1)$, so
$\pi^*\cO_{X'}(1)\cong L^m$.
Then:
\begin{itemize}
\item[(i)] The morphism $\pi\colon X\to X'$ is proper and
induces an isomorphism $\cO_{X'}\eqto\pi_*\cO_X$.
\item[(ii)] For all $n\gg 0$,
the line bundle $\cO_{X'}(n)$ is relatively very ample over $S$,
and for any positive integer $n$ such that $\cO_{X'}(n)$ is
relatively very ample, the morphism
$X\to \Proj(\Sym^\bullet\varphi_*(L^{mn}))$
determined by the direct image of $L^{mn}$
factors uniquely through $X'$, identifying $X'$ with the image of $X$ in
$\Proj(\Sym^\bullet\varphi_*(L^{\otimes n}))$.
\item[(iii)] With the property in \emph{(ii)} the scheme $X'$, with
$\pi\colon X\to X'$, is determined uniquely up to canonical isomorphism,
independently of the choice of $m$, coherent sheaf $\cE$, and morphism
$\varphi^*\cE\to L^m$, and the powers of $L$ that are isomorphic to the
pull-back of a line bundle from $X'$ are precisely all multiples of the
exponent of $M(X/S,L)$.
\item[(iv)] If $\varphi$ is flat and
there is a positive integer $n$ such that for all $i>0$,
\[ R^i\varphi_*(L^n)=R^i\varphi_*(L^{2n})=\dots=0, \]
then the corresponding direct image sheaves are locally free,
their formation commutes with arbitrary base change $S'\to S$,
when $S'$ is Noetherian the
pull-back of $L$ to $S'\times_SX$ determines the contraction
$S'\times_SX\to S'\times_SX'$, and $X'$ is flat over $S$.
\end{itemize}
\end{prop}

\begin{proof}
As in Proposition \ref{prop.semiampleoverfield},
(i) recalls a basic property of
the Stein factorization.
For (ii), we have as before the first statment by standard facts on very ample
line bundles.
Letting $\rho\colon X'\to S$ denote the structure morphism, using the
isomorphism $\cO_{X'}\cong \pi_*\cO_X$ we have
$\varphi_*(L^{\otimes n})\cong \rho_*(\cO_{X'}(n))$, and this determes compatible
morphisms to its projectivization from $X$ and from $X'$, the latter a closed immersion.
For (iii) the argument is as in Proposition \ref{prop.semiampleoverfield}.
The statements concerning the direct image sheaves in (iv) follow from
cohomology and base change, and the flatness from
\[ X'\cong \Proj(\cO_S\oplus \varphi_*(L^{\otimes n})\oplus \varphi_*(L^{\otimes 2n})\oplus \dots). \]
The remaining assertion follows from these facts, using that
some multiple of $n$
satisfies the condition in (ii) both for $X$ and for $S'\times_SX$.
\end{proof}

\begin{prop}
\label{prop.semiampleoverStoscheme}
Let $S$ be a Noetherian scheme, and
let $X$ be a projective scheme over $S$ (resp.\ a projective flat scheme over $S$) with
relatively ample line bundle $L_0$ and effective Cartier divisor $E\subset X$.
Suppose for some $m>0$, the line bundle $L_0|_E\otimes (N_{E/X})^m$
over affine open subsets of $S$
is semiample and determines a flat contraction
$\psi\colon E\to F$.
Suppose, furthermore, $N_{E/X}^\vee$ is ample relative to $\psi$ and
on fibers is normally generated with vanishing $H^1$
of all positive powers (resp.\ vanishing $H^i$ for all $i>0$ of all nonnegative powers).
Then $L:=L_0(mE)$ is semiample with exponent of $M(X,L)$
equal to exponent of $M(E,L_0|_E\otimes(N_{E/X})^m)$ over
affine open subsets of $S$ and determines a contraction
$X\cong B\ell_FY\to Y$ of $E$ to $F$
(resp.\ a contraction $X\cong B\ell_FY\to Y$ of $E$ to $F$ with $Y$ flat
over $S$, whose formation commutes with arbitrary base change $S'\to S$ with $S'$ Noetherian)
such that $Y$ is normally flat along $F$
with normal cone
\[ \Spec\Big(\bigoplus_{n\ge 0} \psi_*(N_{E/X}^\vee)^n\Big) \]
\end{prop}

\begin{proof}
We may suppose $S$ is affine.
Now we argue as in the proof of Proposition \ref{prop.semiampleoverfieldtoscheme},
relative over $S$.
In the flat case with the stronger hypothesis of cohomology vanishing we have
\eqref{eqn.semiampleoverfieldtoschemeH1} also for $j=m$, also with
$L_0$ replaced by $L_0^n$ for any $n>0$ and all $0\le j\le mn$, and also
with any $H^i$ ($i>0$) in place of $H^1$.
Hence $H^i(X,L^n)=0$ for all $i>0$ and $n>0$.
Proposition \ref{prop.semiampleoverS}(iii) gives us the desired conclusion.
\end{proof}

\begin{rema}
\label{rem.semiampleoverStoschemei}
Proposition \ref{prop.semiampleoverStoscheme} is a
strict generalization of Proposition \ref{prop.semiampleoverfieldtoscheme},
which in turn generalizes Proposition \ref{prop.semiampleoverfieldtopoint};
the inclusion of all three results serves expository purposes only.
\end{rema}

\begin{rema}
\label{rem.semiampleoverStoschemeii}
In Proposition \ref{prop.semiampleoverfieldtopoint}, respectively
\ref{prop.semiampleoverfieldtoscheme} and \ref{prop.semiampleoverStoscheme},
the morphism $\pi\colon X\to Y$ induces by pullback and direct image an
equivalence of categories between locally free coherent sheaves on $Y$ and
locally free coherent sheaves on $X$ whose restriction to $E$ is free,
respectively isomorphic to $\psi^*$ of a locally free coherent sheaf on $F$.
That $\pi_*$ of such a sheaf $\cF$ is locally free
follows from the Theorem on Formal Functions and observations concerning
$H^0(X,\cI_E^i\cF/\cI_E^j\cF)$, respectively $H^0(U,\cI_E^i\cF/\cI_E^j\cF)$,
analogous to those in
\eqref{eqn.semiampleoverfieldtopointsurj}--\eqref{eqn.semiampleoverfieldtopointmoregenrl},
respectively
\eqref{eqn.semiampleoverfieldtoschemesurj}--\eqref{eqn.semiampleoverfieldtoschememoregenrl}.
\end{rema}

\begin{rema}
\label{rem.semiampleoverStoschemeiii}
In Proposition \ref{prop.semiampleoverStoscheme} (without the assumption
that $X\to S$ is flat and stronger cohomology vanishing from the statement),
$\pi\colon X\to Y$ is characterized by the following universal
property: for any $S$-scheme $Z$, $S$-morphisms $Y\to Z$ are
by composition with $\pi$ in bijection with $S$-morphisms $X\to Z$ whose
restriction to $E$ factors through $F$ (by $\psi$).
(Such a factorization, if it exists, is unique by \cite[VIII.5.1(a)]{SGA1}.)
Indeed, such an $S$-morphism $X\to Z$ induces a unique
compatible map $Y\to Z$ on the level of sets, continuous
since $\pi$ is proper and surjective and hence a subset $U\subset Y$ is
open if and only if $\pi^{-1}(U)\subset X$ is open, and
with $\pi$ being a contraction, there is a unique compatible morphism of
structure sheaves.
In fact, we may allow $Z$ to be an algebraic space;
then we need the observation that given compatible separated \'etale
morphisms of schemes $X'\to X$, $E'\to E$, and $F'\to F$, the functor
on $Y$-schemes that sends $\widetilde{Y}\to Y$, determining
$\widetilde{X}\to X$, $\widetilde{E}\to E$, and $\widetilde{F}\to F$ by base change,
to the set of compatible lifts $\widetilde{X}\to X'$, $\widetilde{E}\to E'$, and $\widetilde{F}\to F'$,
is representable by a scheme $Y'$, separated and \'etale over $Y$.
Already the functor sending $\widetilde{Y}\to Y$ to the set of
lifts $\widetilde{F}\to F'$ is representable by an algebraic space,
\'etale over $Y$ \cite[\S 1.1]{mustatamustata} \cite[\S 5.1]{alperkresch},
so we may suppose $F'=F$ (and $E'=E$).
Then the observation follows from a general result of Artin on contractions
\cite[Thm.\ 3.1]{artinalgebraizationII} and the property
of separated \'etale morphisms of algebraic spaces
that if the target is a scheme then so is the source \cite[Cor.\ II.6.17]{knutson}.
The universal property in this level of generality
compares nicely with earlier results, e.g., in \cite{mazurcontractions}.
\end{rema}

\bibliographystyle{plain}
\bibliography{pn}

\end{document}